%% file: MainArxiv.tex
\documentclass[11pt, a4paper]{amsart}
\usepackage{hyperref}
\usepackage[papersize={18cm,29cm},total={12.7cm,22.5cm},top=3cm,left=2.7cm]{geometry} 
\usepackage{amssymb, array, amsmath, amscd, subcaption, enumerate, amsthm, fixltx2e, setspace, mathtools, cases}
\usepackage[utf8]{inputenc}
\usepackage{amssymb}
\usepackage{bm}
\usepackage{graphicx} 

\newcommand{\LL}{\mathcal{L}}
\newcommand{\A}{\mathcal{A}}
\newcommand{\B}{\mathcal{B}}
\newcommand{\N}{\mathcal{N}}
\newcommand{\doo}[2]
{\frac{\partial #1}{\partial #2}}
\newcommand{\lapl}[1]
{\Delta #1}
\newcommand{\RNum}[1]{\uppercase\expandafter{\romannumeral #1\relax}}

\newtheorem{Theorem}{Theorem}[section]
\newtheorem{Lemma}[Theorem]{Lemma}
\newtheorem{Proposition}[Theorem]{Proposition}

\theoremstyle{definition}

\newtheorem{assumption}[Theorem]{Assumption}
\newtheorem{remark}[Theorem]{Remark}

\begin{document}

\title[Robust Output Regulation of the Linearized Boussinesq]{Robust Output Regulation of the Linearized Boussinesq Equations with Boundary Control and Observation}

\thispagestyle{plain}

\author{Konsta Huhtala, Lassi Paunonen and Weiwei Hu}
\address{(K. Huhtala and L. Paunonen) Mathematics and Statistic, Faculty of Information Technology and Communication Sciences, Tampere University, PO.\ Box 692, 33101 Tampere, Finland}
\email{konsta.huhtala@tuni.fi, lassi.paunonen@tuni.fi}
\address{(W. Hu) Department of Mathematics, University of Georgia, Athens, GA 30602, USA}
\email{Weiwei.Hu@uga.edu}

\begin{abstract}
	We study temperature and velocity output tracking problem for a two-dimensional room model with the fluid dynamics governed by the linearized translated Boussinesq equations. Additionally, the room model includes finite-dimensional models for actuation and sensing dynamics, thus the complete model dynamics are governed by an ODE-PDE-ODE system.
	As the main result, we design a low-dimensional internal model based controller for robust output racking of the room model.
	Efficiency of the controller is demonstrated through a numerical example of velocity and temperature tracking.
\end{abstract}
\keywords{Linear control systems, robust output regulation, partial differential
equations
\newline The research was supported by the Academy of Finland Grant
number 310489 held by L. Paunonen. L. Paunonen was funded by
the Academy of Finland Grant number 298182. W. Hu was partially
supported by the NSF grant DMS-1813570.} 

\maketitle

\input{Introduction}
\input{BAS}
\input{StabDet}
\input{CLCD}
\input{Example}
\input{Conclusion}

\bibliographystyle{plain}

\end{document}

%% file: Introduction.tex
\section{Introduction}

We consider temperature and flow control for a two-dimensional room model. In the model, behavior of the fluid within the room is described by the Boussinesq equations. These equations couple the fluid flow dynamics given by the Navier--Stokes equations to the fluid temperature dynamics governed by the advection--diffusion equation, and they are commonly used for modeling non-isothermal flows, see e.g. \cite{Borggaard2012,Aulisa2016,Burns2016}.
In this paper, we consider the linearized Boussinesq equations.
As the main control problem, we study output tracking for the room model, where a mix of observations on the fluid temperature and velocity must converge to a desired reference trajectory over time, i.e.
\begin{equation}\label{Eq:convergence}
\Vert y(t)-y_{ref}(t)\Vert \to 0 \qquad \textrm{as} \quad t \to \infty,
\end{equation}
where $y(t) \in \mathbb{R}^p$ is the observation and $y_{ref}(t) \in \mathbb{R}^p$ is the reference output. The considered reference outputs are of the form
\begin{equation}\label{Eq:refsig}
y_{ref}(t) = a_0(t) + \sum_{i=1}^{q_s}a_i(t)\cos(\omega_i t) + b_i(t)\sin(\omega_i t),
\end{equation}
where $0 = \omega_0 < \omega_1 < \cdots < \omega_{q_s}$ are known frequencies and $a_i(t),b_i(t) \in \mathbb{R}^p$ are polynomial vectors with possibly unknown coefficients but known maximal degrees. As the main contribution of this paper, we design a finite-dimensional controller for output tracking of the room model with the room geometry depicted in Figure \ref{fig:Room}.
\begin{figure}[h]
	\begin{center}
		\centering
		\captionsetup{justification=centering}
		\includegraphics[width=5.5cm,height=5cm]{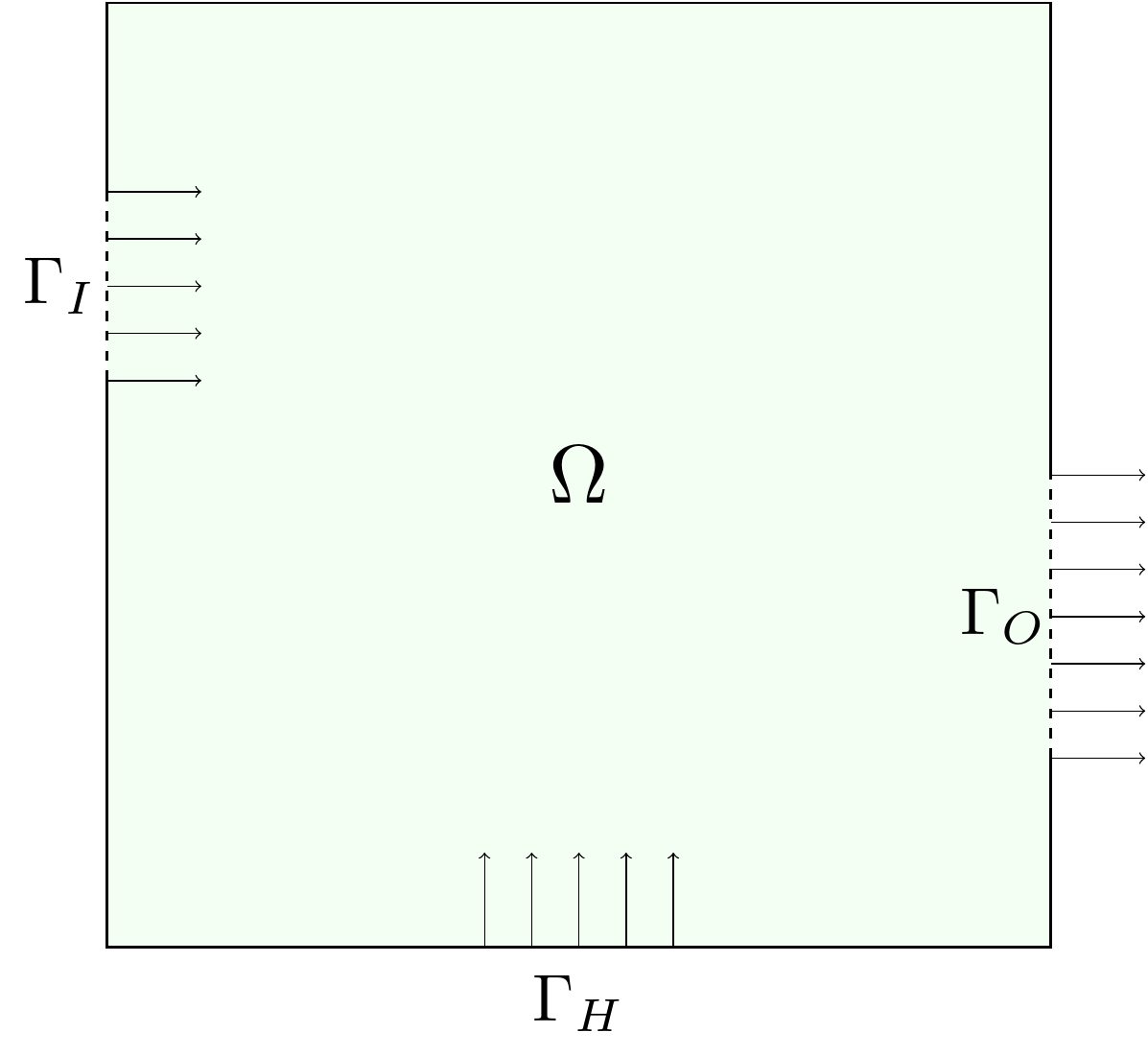}
		\caption{A room with the boundary regions of interest highlighted} 
		\label{fig:Room}
	\end{center}
\end{figure}

The fluid dynamics within the room in Figure \ref{fig:Room} are governed by the linearized translated Boussinesq equations. We focus on a control setup typical for rooms, where the physical control inputs act on the fluid near the boundary of the room, i.e. the walls, the floor or the roof, c.f. \cite{Borggaard2012,Burns2016}. In the model, the fluid flows into and out of the room through the boundary regions $\Gamma_I$ and $\Gamma_O$. Both the fluid velocity and the fluid temperature are controlled within $\Gamma_I$, c.f. \cite{Burns2016}. Additionally, the fluid temperature is controlled within $\Gamma_H$, but no velocity control is applied within $\Gamma_H$ and there is no fluid flow through this boundary section. Observations on the fluid are performed both within the boundary regions and inside the spatial domain. In addition to the fluid dynamics, the room model includes  finite-dimensional dynamical models for the \emph{actuators} and the \emph{sensors} related to the fluid control and observation, respectively. As such, the complete room model is an ODE-PDE-ODE model. Compared to a model with PDE dynamics only, the room model has more complex dynamics but the control and observation operations are bounded.
Additionally, increased model accuracy due to dynamic actuator modeling has been shown in \cite{Zimmer2003} for an acoustic model, and dynamic actuator modeling has been argued to be a more realistic approach to system modeling in general \cite{Burns2012a}. From now on, we will refer to the full ODE-PDE-ODE room model as the \emph{cascade system}.

We achieve the output convergence \eqref{Eq:convergence} for the room model by implementing a controller introduced in \cite{Paunonen2019}. The controller is based on the \emph{internal model principle}, see \cite{Francis1975,Davison1976,Paunonen2010}, and has several desirable properties. It can be used in control of unstable systems, which is essential for this paper due to the fact that the linearized Boussinesq equations may be unstable \cite{Burns2016} (depending on the room geometry and physical parameters). The controller does not require complete state information of the system but rather only uses the observation $y(t)$, and since the controller is based on a Galerkin approximation of the room model combined with model reduction, it is of low-order for fast computations. Finally, the controller is \emph{robust} in that it tolerates small system uncertainties and rejects disturbance signals of the form
\begin{equation}\label{Eq:distsig}
u_{d}(t) = c_0(t) + \sum_{i=1}^{q_s} c_i(t)\cos(\omega_i t) + d_i(t)\sin(\omega_i t),
\end{equation}
which can be applied either within the boundary or inside the spatial domain of the room.
Here $\omega_i$ are the same frequencies as in \eqref{Eq:refsig} and $c_i(t),d_i(t) \in \mathbb{R}^d$ are polynomial vectors with possibly unknown coefficients but known maximal degrees. For linear systems, also several alternative output tracking controllers have been designed. However, none of the other controller options meets all of the above criteria, which motivates our choice of the controller in \cite{Paunonen2019}. Notably, the control solutions in \cite{Byrnes2000,Deutscher2011,Xu2017} lack the robustness property of fault tolerance and disturbance rejection. At the same time, the robust controllers presented in \cite{Hamalainen2000,Rebarber2003,Hamalainen2010,Paunonen2016,Paunonen2019} are either designed for stable systems only or are infinite-dimensional.

Most of the previous work regarding control of the Boussinesq equations focuses on stabilization \cite{Wang2003,Burns2016,Hu2016,Ramaswamy2019}. Examples of output tracking without robustness and based on state feedback have been considered in \cite{Aulisa2016}. Additionally, robust output tracking for a simplified room model with only temperature dynamics and in-domain control and observation has been studied in \cite{Huhtala2021}.
Finally, addition of the actuator dynamics for classes of linear systems has previously been considered in \cite{Burns2012a,Burns2014a,Morris2018}.
In this work, we formulate the abstract system presentation for the cascade system, i.e. with temperature, velocity and ODE dynamics. To guarantee the output convergence \eqref{Eq:convergence}, we also study effects of the additional ODE dynamics on exponential stabilizability and detectability of the room model. Furthermore, we have to carefully select the approximation method for the cascade system during controller construction due to incompressibility of the velocity field.

The paper is organized as follows.
In Section \ref{Sec:system}, we first present the complete ODE-PDE-ODE room model. We formulate the cascade system as an abstract linear control system in Section \ref{ssec:abstractsys}. Section \ref{Ssec:StabDet} is devoted for stabilizability and detectability analysis of the cascade system in terms of us presenting sufficient conditions for these properties.
In Section \ref{Sec:ROR}, we couple the room model with an error feedback controller to guarantee the output convergence \eqref{Eq:convergence}. The controller is based on \cite{Paunonen2019} and we present the design process for the cascade system.
In Section \ref{Sec:example}, we present a numerical example of robust output tracking for the boundary controlled linearized Boussinesq equations with a mix of boundary and in-domain observations and including actuator and sensor dynamics. Finally, the paper is concluded in Section \ref{Sec:conclusion}.

We use the following notation. For a linear operator $A$, $D(A)$ and $\mathcal{N}(A)$ denote its domain and kernel, respectively. Furthermore, the spectrum of $A$ is denoted by $\sigma(A)$ and the resolvent set by $\rho(A)$. The set of bounded linear operators from $X$ to $Y$ is denoted by $\LL(X,Y)$. Finally, $\langle \cdot, \cdot \rangle_\Omega$ and $\langle \cdot, \cdot \rangle_\Gamma$ denote the $L^2$-inner product or duality pairing on the two-dimensional domain $\Omega$ and on the one-dimensional domain $\Gamma$, respectively.

%% file: BAS.tex
\section{The Room Model}\label{Sec:system}

We consider a two-dimensional model of a room already depicted in Figure \ref{fig:Room} with interior $\Omega$ and boundary $\Gamma$. The room has two disjoint vents; an inlet $\Gamma_I$ and an outlet $\Gamma_O$. We denote walls of the room by $\Gamma_W = \Gamma \setminus (\Gamma_I \cup \Gamma_O)$, and assume ``no-slip'' velocity condition at the walls. Regarding temperature, we assume that there is a radiative heater located within $\Gamma_H \subset \Gamma_W$ and on $\Gamma_W \setminus \Gamma_H$ the temperature is fixed. In addition to the radiator, the fluid flow and fluid temperature within the room are affected by Robin boundary control within the inlet. Finally, the fluid is assumed to be stress-free with unforced heat flux within the outlet.

We next formulate the linearized Boussinessq equations around a steady state solution of the Boussinesq equations. The linearized Boussinesq equations are used to describe flow and temperature evolution of the fluid within the room. The PDE is coupled with abstract ODE systems governing the actuation and sensing dynamics. Finally, the coupled ODE-PDE-ODE system is reformulated in an abstract form for which we consider stabilizability and detectability properties. Note that later in the paper we will utilize the abstract formulation for controller implementation.

\subsection{The Linearized Translated Boussinesq Equations with Actuation and Sensing}
The Boussinesq equations for $\xi \in \Omega$ are given by
\begin{subequations}\label{Eq:Bouss}
	\begin{align}
	\dot{w}(\xi,t) &= \frac{1}{Re}\lapl{w}(\xi,t) - (w(\xi,t) \cdot \nabla )w(\xi,t) - \nabla q(\xi,t) \\
	& \quad + \hat{e}_2\frac{Gr}{Re^2}T(\xi,t) + f_{w}(\xi), \nonumber \\
	\dot{T}(\xi,t) &= \frac{1}{RePr} \lapl{T}(\xi,t) - w(\xi,t) \cdot \nabla T(\xi,t) + f_T(\xi), \\
	0 &= \nabla \cdot w(\xi,t), \qquad w(\xi,0) = w_0(\xi), \qquad T(\xi,0) = T_0(\xi),
	\end{align}
\end{subequations}
where $T$ is the temperature, $q$ is the pressure, and $w = [w_{1}, \ w_{2}]^T$ is the velocity of the fluid. Functions $f_w$ and $f_T$ represent a body force and a heat source, respectively, and $\hat{e}_2 = [0,\ 1]^T$ indicates the direction of buoyancy. Finally, $Re$ is the Reynold's number, $Gr$ is the Grashof number and $Pr$ is the Prandtl number. By first linearizing the above equations and then formulating them around a steady state solution $(w_{ss},q_{ss},T_{ss})$ of \eqref{Eq:Bouss} using the change of variables  $v(\xi,t) = w(\xi,t)-w_{ss}(\xi)$, $\theta(\xi,t) = T(\xi,t)-T_{ss}(\xi)$, $p(\xi,t) = q(\xi,t)-q_{ss}(\xi)$, we arrive at the linearized translated Boussinesq equations
\begin{subequations}\label{Eq:lintransbous}
	\begin{align}
	\dot{v}(\xi,t) &= \frac{1}{Re}\lapl{v}(\xi,t) -(w_{ss}(\xi)\cdot \nabla) v(\xi,t) - (v(\xi,t) \cdot \nabla)w_{ss}(\xi) \\
	 &\quad - \nabla p(\xi,t) + \hat{e}_2 \frac{Gr}{Re^2} \theta(\xi,t), \nonumber \\
	\dot{\theta}(\xi,t) &= \frac{1}{RePr} \lapl{\theta}(\xi,t) - w_{ss}(\xi) \cdot \nabla \theta(\xi,t) - v(\xi,t) \cdot \nabla T_{ss}(\xi), \\
	0 &= \nabla \cdot v(\xi,t), \qquad v(\xi,0) = v_0(\xi), \qquad \theta(\xi,0) = \theta_0(\xi).
	\end{align}	
The considered boundary conditions are
	\begin{align}
	&\big( \mathcal{T}(v(\xi,t),p(\xi,t)) \cdot n + \alpha_vv(\xi,t) \big)\big|_{\Gamma_I} = \begin{bmatrix}
	b_{v}(\xi) & b_{dv}(\xi)
	\end{bmatrix}\begin{bmatrix}
	u_{bv}(t) \\ u_{dv}(t)
	\end{bmatrix}, \label{Eq:boundarycondtv} \\
	&\bigg( \frac{1}{RePr}\doo{\theta}{n}(\xi,t) + \alpha_\theta \theta(\xi,t)  \bigg) \big|_{\Gamma_I} = \begin{bmatrix}
	b_{\theta_I}(\xi) & b_{d\theta_I}(\xi)
	\end{bmatrix} \begin{bmatrix}
	u_{b\theta_I}(t) \\ u_{d\theta_I}(t)
	\end{bmatrix},\label{Eq:boundarycondtthi} \\
	&\bigg(\frac{1}{RePr} \doo{\theta}{n}(\xi,t)\bigg)\big|_{\Gamma_H} = \begin{bmatrix}
	b_{\theta_H}(\xi) & b_{d\theta_H}(\xi)
	\end{bmatrix} \begin{bmatrix}
	u_{b\theta_H}(t) \\ u_{d\theta_H}(t)
	\end{bmatrix}, \label{Eq:boundarycondtthh} \\
	&\big( \mathcal{T}(v(\xi,t),p(\xi,t)) \cdot n \big)|_{\Gamma_O} = 0, \qquad v(\xi,t)|_{\Gamma_W} = 0, \label{Eq:boundarycondvr} \\
	&\doo{\theta}{n}(\xi,t)|_{\Gamma_O} = 0, \qquad \theta(\xi,t)|_{(\Gamma_W \setminus \Gamma_H)} = 0, \label{Eq:boundarycondtr}
	\end{align}
\end{subequations}
where $n$ denotes the unit outward normal vector of $\Gamma$, $\mathcal{T}$ is the fluid Cauchy stress tensor and $\alpha_v,\alpha_\theta \ge 0$ are constants, $u_b = [u_{bv}, \ u_{b\theta_I}, \ u_{b\theta_H}]$ are control inputs, $u_d = [u_{dv}, \ u_{d\theta_I}, \ u_{d\theta_H}]^T$ are disturbance inputs and the control and the disturbance inputs are applied via the shape functions $b_{v}$, $b_{\theta_I}$, $b_{\theta_H}$, $b_{dv}$, $b_{d\theta_I}$ and $b_{d\theta_H}$.
The inputs $u_b(t)$ are not directly generated by the controller, but are rather given as the output of the finite-dimensional actuator
\begin{subequations}\label{Eq:actuator}
	\begin{align}
	\dot{x}_a(t) &= A_ax_a(t) + B_au(t), \qquad x_a(0) = x_{a0} \in \mathbb{R}^{n_a}, \\
	u_b(t) &= C_ax_a(t),
	\end{align}
\end{subequations}
which takes as its input the control signal $u(t)$ generated by the controller.

We are mainly interested in two types of observations. These are weighted temperature or velocity averages either over a two-dimensional domain $\Omega_C \subset \Omega$, given by
\begin{subequations}\label{Eq:boussobs}
\begin{equation}\label{Eq:gammaobs}
y_{\Omega}(t) = \bigg\langle c_\Omega(\xi),\begin{bmatrix}
v(\xi,t) &  \theta(\xi,t)
\end{bmatrix}^T \bigg\rangle_{\Omega_C},
\end{equation}
or over a one-dimensional domain $\Gamma_C \subset \Gamma$, given by
\begin{equation}\label{Eq:omegaobs}
y_\Gamma(t) = \bigg\langle c_\Gamma(\xi),\begin{bmatrix}
v(\xi,t) &  \theta(\xi,t)
\end{bmatrix}^T \bigg\rangle_{\Gamma_C},
\end{equation}
\end{subequations}
and we denote by $y_b$ the observation of interest consisting of a combination of the two types. Note that one may include several observations of one type with the restriction that we need to increase the number of inputs $u(t)$ accordingly to at least match the number of observations, c.f. Assumption \ref{ass:invariant} in Section \ref{Sec:ROR}. These additional inputs are included in \eqref{Eq:boundarycondtv}-\eqref{Eq:boundarycondtthh} by considering vector valued $u_{bv}$, $u_{b\theta_I}$, $u_{b_{\theta_H}}$, $b_{v}$, $b_{\theta_I}$ and $b_{\theta_H}$.
Just as in the case of system input, the system output is also processed by a finite-dimensional system, the sensor
\begin{subequations}\label{Eq:sensor}
	\begin{align}
	\dot{x}_s(t) &= A_sx_s(t) + B_s y_b(t), \qquad x_s(0) = x_{s0} \in \mathbb{R}^{n_s}, \\
	y(t) &= C_s x_s(t)
	\end{align}
\end{subequations}
with the observation $y(t)$. The complete plant thus consists of the linearized translated Boussinesq equations \eqref{Eq:lintransbous} coupled with the actuator \eqref{Eq:actuator} via the input $u_b$ and with the sensor \eqref{Eq:sensor} via the output $y_b$. As already mentioned, we refer to the system \eqref{Eq:lintransbous}-\eqref{Eq:sensor} as the cascade system. Figure \ref{fig:Contscheme} depicts the control scheme consisting of the cascade system and an \emph{error feedback controller}.
\begin{figure}[h!]
	\begin{center}
		\centering
		\captionsetup{justification=centering}
		\includegraphics[scale=1.5]{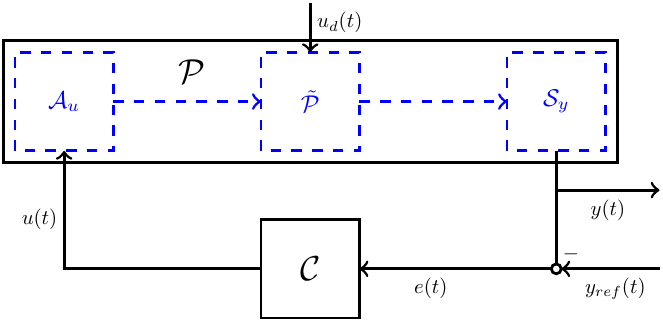}
		\caption{A closed-loop control scheme with an actuator $\mathcal{A}_u$, a plant $\tilde{\mathcal{P}}$, a sensor $\mathcal{S}_y$, the cascade system $\mathcal{P}$ and a controller $\mathcal{C}$}
		\label{fig:Contscheme}
	\end{center}
\end{figure}

The control goal is considered for the observation $y(t)$ of the sensor, which we want to converge exponentially to a prescribed reference trajectory $y_{ref}(t)$ of the form \eqref{Eq:refsig} despite the disturbance signal $u_d(t)$ given by \eqref{Eq:distsig}, i.e. for some $M_r,\omega_r > 0$ it should hold that
\begin{equation}\label{Eq:desiredtracking}
	\Vert y(t) - y_{ref}(t) \Vert \le M_re^{-\omega_rt}P_0.  
\end{equation} 
Here $P_0$ is determined by the initial states of the linearized translated Boussinesq equations, the actuator, the sensor and the controller and the coefficients $a_i(t)$, $b_i(t)$, $c_i(t)$ and $d_i(t)$ of the reference signal \eqref{Eq:refsig} and the disturbance signal \eqref{Eq:distsig}.

\subsection{Abstract Formulation of the Control System}\label{ssec:abstractsys}
The controller to be implemented achieves output tracking for a class of abstract linear systems, which motivates us to present the cascade system as one. The system dynamics operator will be defined using a weak formulation of the cascade system. We then follow up with the formulation of operators related to \emph{boundary control system} representation of the cascade system, a presentation choice natural in the presence of boundary system inputs \eqref{Eq:boundarycondtv}-\eqref{Eq:boundarycondtthh}.

To prepare for the formulations, we define the spaces
\begin{subequations}
	\begin{align*}
	X_v &= \big\{v \in (L^2(\Omega))^2 \big| \nabla \cdot v = 0, \  (v \cdot n) |_{\Gamma_W} = 0 \big\}, \\
	X_{b} &= X_v \times L^2(\Omega), \\
	X_\Gamma &= (L^2(\Gamma_I))^2 \times L^2(\Gamma_I) \times  L^2(\Gamma_H), 
	\\
	H_v &= \big\{ v \in (H^1(\Omega))^2 \big|\ \nabla \cdot v = 0, \  v|_{\Gamma_W} = 0 \big\}, \\	
	H_{\theta} &= \big\{ \theta \in H^1(\Omega) \big|\ \theta|_{(\Gamma_W\setminus \Gamma_H)} = 0 \big\}, \\
	H_{b} &= H_v \times H_\theta
	\end{align*}
	concerning the Boussinesq equations and the spaces
	\begin{align}
	&X = X_b \times \mathbb{R}^{n_a} \times \mathbb{R}^{n_s}, \label{Eq:statespace} \\
	&H = H_b \times \mathbb{R}^{n_a} \times \mathbb{R}^{n_s}
	\end{align}
\end{subequations}
concerning the cascade system.
Furthermore, for all $x = [x_b, \ x_a, \ x_s]^T \in X$, where $x_b = [v, \ \theta]^T$, we define the norms
\begin{subequations}\label{Eq:norms}
\begin{align}
&\Vert x\Vert_X^2 = \Vert x_b\Vert_{X_b}^2 + \Vert x_a\Vert_{\mathbb{R}^{n_a}}^2 + \Vert x_s\Vert_{\mathbb{R}^{n_s}}^2, \\
&\Vert x\Vert_H^2= \Vert x_b\Vert_{H_b}^2 + \Vert x_a\Vert_{\mathbb{R}^{n_a}}^2 + \Vert x_s\Vert_{\mathbb{R}^{n_s}}^2,
\end{align}
\end{subequations}
and denote the input space $U = \mathbb{R}^m$, the output space $Y = \mathbb{R}^p$ and the disturbance space $U_d = \mathbb{R}^d$.

The presented observations \eqref{Eq:boussobs} are not the only possible choices, and before focusing on the system as a whole we present an assumption characterizing the class of suitable observations
\begin{equation}\label{Eq:BoussObs}
y_b(t) = C_bx_b(t)
\end{equation}
on the linearized Bousinesq equations.
\begin{assumption}\label{ass:admobs}
	The observation operator satisfies $C_b \in \LL(H_b,Y_b)$ for some $Y_b \coloneqq \mathbb{R}^{p_b}$.
\end{assumption}

\begin{Lemma}\label{lemma:C}
	For system \eqref{Eq:lintransbous}, both the observation $y_\Omega = \langle c_\Omega,x_b \rangle_{\Omega_C}$ and $y_\Gamma = \langle c_\Gamma,x_b \rangle_{\Gamma_C}$ in \eqref{Eq:boussobs} with $c_\Omega \in (L^2(\Omega_C))^2\times L^2(\Omega_C)$, $c_\Gamma \in (L^2(\Gamma_C))^2\times L^2(\Gamma_C)$ satisfy Assumption \ref{ass:admobs}.
\end{Lemma}
\begin{proof}
Clearly $\langle c_\Omega, \cdot \rangle_{\Omega_C} \in \LL(X_b,\mathbb{R})$. Due to properties of the trace operator, we have
\begin{align*}
\langle c_\Gamma,x_b \rangle_{\Gamma_C} \le \Vert c_\Gamma \Vert_{L^2(\Gamma_C)}\Vert x_b \Vert_{L^2(\Gamma_C)} \le k \Vert x_b \Vert_{H^{1/2}(\Omega)},
\end{align*}
thus $\langle c_\Gamma, \cdot \rangle_{\Gamma_C} \in \LL(H_b,\mathbb{R})$.
\end{proof}
\noindent Existence and uniqueness of steady state solutions for the Boussinesq equations are outside the scope of this work. For the following analysis of the cascade system, we assume that a weak steady state solution
\begin{equation*}
(w_{ss},q_{ss},T_{ss}) \in H_v \times L^2(\Omega) \times H_\theta,
\end{equation*}
for the Boussinesq equations \eqref{Eq:Bouss} exists.
Discussion on existence and uniqueness of the steady state solutions can be found in e.g. \cite{Kim2012}.

As the first step towards abstract formulation of the cascade system, we construct the system dynamics operator $A$ via a weak formulation of the cascade system and verify that it generates a strongly continuous semigroup on $X$. To that end, we define the bilinear and trilinear forms
\begin{subequations}\label{Eq:bilin}
	\begin{align}
	&a_v(v,\psi) = \frac{2}{Re}\langle \epsilon (v),\epsilon(\psi) \rangle_{\Omega} + \alpha_v\langle v,\psi  \rangle_{\Gamma_I} \qquad \forall v,\psi \in H_v, \label{Eq:Stokesbilin} \\
	&a_\theta(\theta,\phi) = \frac{1}{RePr}\langle \nabla \theta,\nabla \phi\rangle_{\Omega} + \alpha_\theta\langle \theta,\phi \rangle_{\Gamma_I} \qquad \forall \theta,\phi \in H_\theta, \label{Eq:Diffusionbilin} \\
	&b_v(v_1,v_2,\psi) = \langle (v_1 \cdot \nabla)v_2,\psi \rangle_{\Omega} \qquad \forall v_1,v_2,\psi \in H_v, \\ 
	&b_\theta(v,\theta,\phi)= \langle v \cdot \nabla \theta,\phi \rangle_{\Omega} \qquad \forall v \in X_v,\ \forall \theta,\phi \in H_\theta,\\
	&b_0(\theta,\psi) = \bigg\langle \hat{e}_2 \frac{Gr}{Re^2}\theta,\psi \bigg\rangle_{\Omega} \qquad \forall \theta \in L^2(\Omega), \ \forall \psi \in (L^2(\Omega))^2.
	\end{align}
\end{subequations}
Here
\begin{align*}
\epsilon(v) = \frac{1}{2}\big(\nabla v + (\nabla v)^T\big) \qquad \forall v \in (H^1(\Omega))^2,
\end{align*}
thus the Cauchy stress tensor is given by
\begin{equation*}
	\mathcal{T}(v,p) = \frac{2}{Re}\epsilon(v) - pI.
\end{equation*}
Note that $0 = -a_\theta(\theta,\phi)$ corresponds to a weak formulation of the stationary diffusion equation for the temperature subject to \eqref{Eq:boundarycondtthi}, \eqref{Eq:boundarycondtthh} and \eqref{Eq:boundarycondtr} with zero control and disturbance, i.e. when
\begin{equation*}
	\bigg(\frac{1}{RePr}\doo{\theta}{n} + \alpha_\theta \theta\bigg)\big|_{\Gamma_I} = 0, \qquad \bigg(\frac{1}{RePr}\doo{\theta}{n}\bigg)\big|_{\Gamma_H} = 0.
\end{equation*}
Similarly, due to Stokes formula and incompressibility, $0 = -a_v(v,\psi)$ corresponds to a weak formulation of the stationary Stokes equation subject to \eqref{Eq:boundarycondtv} and \eqref{Eq:boundarycondvr} with zero control and disturbance, i.e. when
\begin{equation*}
	\big(\mathcal{T}(v,p)\cdot n + a_vv \big)\big|_{\Gamma_I} = 0.
\end{equation*}

Consider the cascade system \eqref{Eq:lintransbous}-\eqref{Eq:sensor} subject to a constant boundary disturbance signal $u_d' = [u_{dv}',\ u_{d\theta_I}',\ u_{d\theta_H}']^T$ and denote $g_{dv} = b_vu_{bv}'$, $g_{d\theta_I} = b_{\theta_I}u_{d\theta_I}'$, $g_{d\theta_H} = b_{\theta_H}u_{d\theta_H}'$.
Now the boundary conditions \eqref{Eq:boundarycondtv}-\eqref{Eq:boundarycondtthh} for the cascade system are
\begin{align*}
\big(\mathcal{T}(v,p)\cdot n \big)\big|_{\Gamma_I} &= -a_vv\big|_{\Gamma_I} + b_vC_{a_v}x_a  + g_{d_v}, \\
\bigg(\frac{1}{RePr}\doo{\theta}{n}\bigg)\big|_{\Gamma_I} &= - \alpha_\theta \theta\big|_{\Gamma_I} + b_{\theta_I}C_{a_{\theta_I}}x_a + g_{d_{\theta_I}}, \\ \bigg(\frac{1}{RePr}\doo{\theta}{n}\bigg)\big|_{\Gamma_H} &= b_{\theta_H}C_{a_{\theta_H}}x_a + g_{d_{\theta_H}},
\end{align*}
where $C_{a_v}$, $C_{a_{\theta_I}}$ and $C_{a_{\theta_H}}$ are obtained from $C_a = [C_{a_v}, \ C_{a_{\theta_I}}, \ C_{a_{\theta_H}}]^T$.
A weak formulation for a steady state solution of the cascade system subject to $u_d'$ and a constant control input $u'$ is now given by
\begin{align*}
0 &=  - a_v(v,\psi) - b_v(v,w_{ss},\psi) - b_v(w_{ss},v,\psi) + b_0(\theta,\psi) + \langle \mathcal{T}(v,p)\cdot n,\psi \rangle_{\Gamma_I} \\ & \quad +\alpha_v\langle v,\psi \rangle_{\Gamma_I} \\
&= - a_v(v,\psi) - b_v(v,w_{ss},\psi) - b_v(w_{ss},v,\psi) + b_0(\theta,\psi)
\\ & \quad + \langle b_{v} C_{a_v} x_a + g_{dv},\psi \rangle_{\Gamma_I} \qquad \forall \psi \in H_v,\\
0 &= -a_\theta(\theta,\phi) - b_{\theta}(w_{ss},\theta,\phi) - b_{\theta}(v,T_{ss},\phi) + \frac{1}{RePr}\big\langle \doo{\theta}{n},\phi \big\rangle_{\Gamma_I \cup \Gamma_H}
\\ &\quad + \alpha_\theta \langle \theta, \phi \rangle_{\Gamma_I} \\
&= -a_\theta(\theta,\phi) - b_{\theta}(w_{ss},\theta,\phi) - b_{\theta}(v,T_{ss},\phi) + \langle b_{\theta_I}C_{\theta_I}x_a + g_{d\theta_I}, \phi \rangle_{\Gamma_I}  
\\ &\quad+ \langle b_{\theta_H}C_{\theta_H}x_a + g_{d\theta_H} , \phi \rangle_{\Gamma_H} \qquad  \forall \phi \in H_\theta, \\
0 &= \langle A_ax_a,\psi_a \rangle_{\mathbb{R}^{n_a}} + \langle B_a u', \psi_a \rangle_{\mathbb{R}^{n_a}} \qquad \forall \psi_a \in \mathbb{R}^{n_a}, \\
0 &= \langle A_sx_s,\psi_s \rangle_{\mathbb{R}^{n_s}} + \big\langle B_sC_b\begin{bmatrix}
v \\ \theta
\end{bmatrix},\psi_s \big\rangle_{\mathbb{R}^{n_s}} \qquad \forall \psi_s \in \mathbb{R}^{n_s}.
\end{align*}
Motivated by the weak formulation, we define for $u_d' = 0$ and $u'=0$ the bilinear form
\begin{align}\label{Eq:bilina0}
	&a_0(\Phi,\Psi) \\
	&= a_0((v,\theta,x_a,x_s),(\psi,\phi,\psi_a,\psi_s))\nonumber \\
	&= a_v(v,\psi) + a_\theta(\theta,\phi) + b_v(v,w_{ss},\psi) + b_v(w_{ss},v,\psi) + b_\theta(w_{ss},\theta,\phi) \nonumber
	\\ & \quad + b_\theta(v,T_{ss},\phi) -b_0(\theta,\psi) - \langle b_v C_{a_v}x_{a},\psi	\rangle_{\Gamma_I} 
	- \langle b_{\theta_I} C_{a_{\theta_I}}x_{a},\phi \rangle_{\Gamma_I} \nonumber
	\\ & \quad - \langle b_{\theta_H}C_{a_{\theta_H}}x_{a},\phi \rangle_{\Gamma_H}  -\big\langle B_sC_b\begin{bmatrix}
	v & \theta
	\end{bmatrix}^T,\psi_s \big\rangle_{\mathbb{R}^{n_s}} - \langle A_ax_a,\psi_a \rangle_{\mathbb{R}^{n_a}} \nonumber \\
	& \quad- \langle A_sx_s,\psi_s \rangle_{\mathbb{R}^{n_s}} \qquad \forall \Psi \in H \nonumber
\end{align}
and more generally for $u'=0$ and some $u_d'\in \mathbb{R}^d$ the bilinear form
\begin{align}\label{Eq:bilinag}
	a_g(\Phi,\Psi) 
	&= a_0((v,\theta,x_a,x_s),(\psi,\phi,\psi_a,\psi_s)) - \langle g_{dv}, \psi \rangle_{\Gamma_I} - \langle g_{d\theta_I},\phi \rangle_{\Gamma_I}  \\ &\quad - \langle g_{d\theta_H},\phi \rangle_{\Gamma_H}. \nonumber
\end{align}
Using the bilinear form $a_0(\cdot,\cdot)$, we define the linear operator $A$ by
\begin{align}\label{Eq:opA}
	\langle Ax,\Psi \rangle_X &= -a_0(x,\Psi),  \\	 
	D(A) &= \big\{x \in H \big|\ \forall \Psi \in H, \ \Psi \to a_0(x,\Psi) \textrm{ is $X$-continuous} \big\}. \nonumber
\end{align}
We note that the geometry of $\Omega$ and the presence of mixed boundary conditions reduce regularity of the solutions of \eqref{Eq:lintransbous} so that $D(A) \not\subset(H^2(\Omega))^2 \times H^2(\Omega) \times \mathbb{R}^{n_a} \times \mathbb{R}^{n_s}$, c.f. \cite{Mazya2009,Burns2016,He2018}.

The following semigroup generation result is not only needed for the abstract system formulation but also the coercivity and boundedness results for the bilinear form will be utilized for the controller implementation to achieve output tracking, c.f. \cite{Paunonen2019}. Note that similar results for both the Boussinesq equations and the Navier--Stokes equations without additional ODE dynamics have been presented in multiple papers, see e.g. \cite{Nguyen2015,Burns2016,He2018}.
\begin{Theorem}\label{thm:bilin}
	Operator $A$ is the generator of an analytic semigroup on $X$ and the bilinear form $a_0(\cdot,\cdot)$ is $H$-bounded and $H$-coercive, i.e. $H$ can be continuously and densely embedded in $X$ and there exist $c,\lambda,\gamma > 0$ such that for all $\Phi,\Psi \in H$
	\begin{align*}
	\vert a_0(\Phi,\Psi) \vert &\le c \Vert \Phi \Vert_H \Vert \Psi \Vert_H, \\
	a_0(\Phi,\Phi) &\ge \gamma \Vert \Phi \Vert_H^2 - \lambda \Vert \Phi \Vert_X^2.
	\end{align*}
\end{Theorem}
\begin{proof}
	Throughout the proof, we denote by $c$ a generic positive constant which may have a different value for each occurrence.
	We start by considering the terms $a_\theta(\cdot,\cdot)$ and $a_v(\cdot,\cdot)$. Now properties of the trace operator imply
	\begin{equation*}
		0 \le \alpha_\theta \langle \theta,\theta \rangle_{\Gamma_I} \le c \Vert \theta \Vert_{H^1}^2,
	\end{equation*}
	thus using Poincare's inequality we get for $\theta \in H_\theta$ and a constant $c_\theta > 0$
	\begin{equation}\label{Eq:diffcoer}
		a_\theta(\theta,\theta) = \frac{1}{RePr}\langle \nabla \theta, \nabla \theta \rangle_{\Omega} + \alpha_\theta \langle \theta,\theta \rangle_{\Gamma_I} \ge c_\theta \Vert \theta \Vert_{H^1}^2, 
	\end{equation}
	i.e. $a_\theta(\cdot,\cdot)$ is $H_\theta$-coercive. Since
	\begin{align}\label{Eq:diffbounded}
		|a_\theta(\theta,\phi)| &\le \bigg|\frac{1}{RePr}\langle \nabla \theta,\nabla \phi \rangle_{\Omega} \bigg| + |\alpha_\theta \langle \theta,\phi \rangle_{\Gamma_I}| \\
		&\le \frac{1}{RePr} \Vert \theta \Vert_{H^1} \Vert \phi \Vert_{H^1} + c \Vert \theta \Vert_{H^1} \Vert \phi \Vert_{H^1}, \nonumber
	\end{align}
	$a_\theta(\cdot,\cdot)$ is also $H_\theta$-bounded. 
	
	Regarding $a_v(\cdot,\cdot)$, it similarly holds that
	\begin{equation*}
		0 \le \alpha_v\langle v,v\rangle_{\Gamma_I} \le c \Vert v \Vert_{H^1}^2,
	\end{equation*}
	and the norm $\Vert \epsilon(\cdot) \Vert_{L^2}$ is equivalent to the norm $\Vert \cdot \Vert_{H^1}$ through Korn's and Poincare's inequalities. Now for $v \in H_v$ and a constant $c_v > 0$
	\begin{equation}\label{Eq:Stokescoer}
		a_v(v,v) = \frac{2}{Re} \langle \epsilon(v)
		,\epsilon(v) \rangle_{\Omega} + \alpha_v \langle v,v \rangle_{\Gamma_I} \ge \frac{2}{Re}\Vert \epsilon(v) \Vert_{L^2}^2 \ge c_v \Vert v \Vert_{H^1}^2,
	\end{equation}
	thus $a_v(\cdot,\cdot)$ is $H_v$-coercive. Since additionally
	\begin{align}\label{Eq:Stokesbdd}
		|a_v(v,\psi)| &\le \bigg| \frac{2}{Re}\langle \epsilon(v),\epsilon(\psi) \rangle_{\Omega}\bigg| + |\alpha_v \langle v,\psi \rangle_{\Gamma_I}|	\\
		&\le c \big( \Vert v \Vert_{H^1} \Vert \psi \Vert_{H^1} +  \Vert v \Vert_{H^1} \Vert \psi \Vert_{H^1} \big), \nonumber
	\end{align}
	$a_v(\cdot,\cdot)$ is also $H_v$-bounded. Combining \eqref{Eq:diffcoer}-\eqref{Eq:Stokesbdd}, we have that there exist constants $c_1,\gamma_1 > 0$ such that for all $\phi_b,\psi_b \in H_b$ the bilinear form
	\begin{equation*}
		a_1(\psi_b,\phi_b) = a_1((v,\theta),(\psi,\phi)) \coloneqq a_v(v,\psi) + a_\theta(\theta,\phi)
	\end{equation*} satisfies
		\begin{subequations}\label{Eq:NSD}
		\begin{align}
		|a_1(\phi_b,\psi_b)| &\le c_1\Vert \phi_b \Vert_{H_b} \Vert \psi_b \Vert_{H_b}, \label{Eq:NSDbound} \\
		a_1(\phi_b,\phi_b) &\ge \gamma_1\Vert \phi_b \Vert_{H_b}^2. \label{Eq:NSDcoer}
		\end{align}
	\end{subequations}
	
	The rest of the proof now consists of presenting estimates for the norms of the remaining terms of $a_0(\cdot,\cdot)$. 
	We immediately have that	
	\begin{align}
		&|b_0(\theta,\psi)| \le c \Vert\theta\Vert_{L^2}\Vert \psi \Vert_{X_v},\label{Eq:Xbounded1} \\
		&|\langle A_ax_a,\psi_a\rangle_{\mathbb{R}^{n_a}}| \le c \Vert x_a \Vert_{\mathbb{R}^{n_a}} \Vert \psi_a \Vert_{\mathbb{R}^{n_a}}, \\
		&|\langle A_sx_s,\psi_s\rangle_{\mathbb{R}^{n_s}}| \le c \Vert x_s \Vert_{\mathbb{R}^{n_s}} \Vert \psi_s \Vert_{\mathbb{R}^{n_s}}. \label{Eq:Xbounded3}
	\end{align}
Regarding the form $b_\theta(\cdot,\cdot,\cdot)$, by Sobolev embeddings, $L^2$-duality of $H^{1/2}$ and $H^{-1/2}$ and Ladyzhenskaya's inequality
	\begin{align}\label{Eq:bthcoer1}
		&|b_\theta(v,T_{ss}, \theta)| = |\langle v \cdot \nabla T_{ss}, \theta \rangle_{\Omega}|\\
		&\le |\langle vT_{ss},\nabla \theta \rangle_{\Omega}| + |\langle v \cdot n,T_{ss}\theta \rangle_{\Gamma}| \nonumber \\ 
		&\le  \Vert vT_{ss} \Vert_{L^2} \Vert \nabla \theta \Vert_{L^2} + c\Vert v \Vert_{H^1} \Vert T_{ss}\theta \Vert_{L^2} \nonumber \\
		&\le \Vert v \Vert_{L^4} \Vert T_{ss} \Vert_{L^4} \Vert \nabla \theta \Vert_{L^2} + c \Vert v \Vert_{H^1} \Vert T_{ss} \Vert_{L^4} \Vert \theta \Vert_{L^4} \nonumber \\
		&\le c \Vert T_{ss} \Vert_{H^1} \big( \Vert v \Vert_{L^2}^{1/2} \Vert \nabla v \Vert_{L^2}^{1/2} \Vert \nabla \theta \Vert_{L^2} + \Vert v \Vert_{H^1} \Vert \theta \Vert_{L^2}^{1/2} \Vert \nabla \theta \Vert_{L^2}^{1/2} \big). \nonumber
	\end{align}
	Similarly, we have
	\begin{align}\label{Eq:bthcoer2}
		|b_\theta(w_{ss},\theta,\theta)| &= \frac{1}{2}\langle w_{ss}, \nabla\theta^2 \rangle_{\Omega} = \frac{1}{2}\langle w_{ss} \cdot n , \theta^2 \rangle_{\Gamma} \le c\Vert w_{ss} \Vert_{H^1} \Vert \theta^2 \Vert_{L^2}\\
		&
		= c\Vert w_{ss} \Vert_{H^1}\Vert \theta \Vert_{L^4}^2 \le c \Vert \theta \Vert_{L^2} \Vert \nabla \theta \Vert_{L^2}. \nonumber
	\end{align}
	Furthermore,
	\begin{align}\label{Eq:bthbounded}
		|b_\theta(v,\theta,\phi)| &\le |\langle v,\nabla(\theta \phi) \rangle_{\Omega}| + |\langle v \theta , \nabla \phi \rangle_{\Omega}| \\
		&= |\langle v \cdot n, \theta \phi \rangle_{\Gamma}| + |\langle v \theta , \nabla\phi \rangle_{\Omega}| \nonumber\\ 
		&\le c \Vert v \Vert_{H^1} \Vert \theta \Vert_{H^1} \Vert \phi \Vert_{H^1}. \nonumber
	\end{align}
	
	Regarding the form $b_v(\cdot,\cdot,\cdot)$, we again use $L^2$-duality of $H^{1/2}$ and $H^{-1/2}$, Sobolev embeddings  and Ladyzhenskaya's inequality to form the estimates. Now
	\begin{align}\label{Eq:bvcoer1}
		|b_v(w_{ss},v,v)| &= |\langle (w_{ss} \cdot \nabla) v,v \rangle_{\Omega}| \\ 
		&\le	|\langle w_{ss},(v \cdot \nabla)v \rangle_{\Omega}| + |\langle w_{ss}\cdot n, v \cdot v \rangle_{\Gamma}| \nonumber \\
		& \le c \big( \Vert w_{ss} \Vert_{L^4} \Vert v \Vert_{L^4} \Vert \nabla v \Vert_{L^2} + \Vert w_{ss} \Vert_{H^1} \Vert v \Vert_{L^4}^2 \big) \nonumber \\ 
		&\le c\big( \Vert v \Vert_{L^2}^{1/2} \Vert \nabla v \Vert_{L^2}^{3/2} + \Vert v \Vert_{L^2} \Vert \nabla v \Vert_{L^2} \big) \nonumber
	\end{align}
	and
	\begin{align}\label{Eq:bvcoer2}
	|b_v(v,w_{ss},v)| &= |\langle (v \cdot \nabla) w_{ss},v \rangle_{\Omega}| \\
	&\le \Vert v \Vert_{L^4}^2\Vert \nabla w_{ss} \Vert_{L^2} \nonumber \\ 
	&\le c \Vert w_{ss} \Vert_{H^1} \Vert v \Vert_{L^2}\Vert \nabla v \Vert_{L^2}. \nonumber
	\end{align}
	Additionally,
	\begin{align}\label{Eq:bvbounded}
	|b_v(v_1,v_2,\psi)| &\le |\langle v_1,(v_2 \cdot \nabla)\psi \rangle_\Omega| + |\langle v_1 \cdot n,v_2 \cdot \psi \rangle_\Gamma|  \\
	&\le c\Vert v_1 \Vert_{L^4} \Vert v_2 \Vert_{L^4} \Vert \nabla \psi \Vert_{L^2} + c \Vert v_1 \Vert_{H^1} \Vert v_2 \cdot \psi \Vert_{L^2} \nonumber \\
	&\le c \big( \Vert v_1 \Vert_{H^1}\Vert v_2 \Vert_{H^1}\Vert \psi \Vert_{H^1} + \Vert v_1 \Vert_{H^1} \Vert v_2 \Vert_{L^4} \Vert \psi \Vert_{L^4} \big) \nonumber \\
	&\le c\Vert v_1 \Vert_{H^1}\Vert v_2 \Vert_{H^1}\Vert \psi \Vert_{H^1}. \nonumber
	\end{align}	
		
	Finally, properties of the trace operator together with Assumption \ref{ass:admobs} and duality imply
	\begin{align}
		&|\langle b_t C_a \psi_a,\psi_b \rangle_{\Gamma}| \le c \Vert \psi_a \Vert_{\mathbb{R}^{n_a}} \Vert\psi_b \Vert_{H_b}, \label{Eq:Bbbound} \\
		&|\langle B_sC_b\psi_b, \psi_s \rangle_{\mathbb{R}^{n_s}}| \le c \Vert \phi_s \Vert_{\mathbb{R}^{n_s}} \Vert \psi_b \Vert_{H_b}, \label{Eq:Cbbound}
	\end{align}
	where $b_t \coloneqq [b_v, \ b_{\theta_I}, \ b_{\theta_H}]$. Recalling the norm definitions \eqref{Eq:norms}, the equations \eqref{Eq:NSDbound}, \eqref{Eq:Xbounded1}-\eqref{Eq:Xbounded3}, \eqref{Eq:bthbounded} and \eqref{Eq:bvbounded}-\eqref{Eq:Cbbound} together imply $H$-boundedness of $a_0(\cdot,\cdot)$. $H$-coercivity of $a_0(\cdot,\cdot)$ follows from \eqref{Eq:NSDcoer} after applying Young's inequality to \eqref{Eq:bthcoer1}, \eqref{Eq:bthcoer2}, \eqref{Eq:bvcoer1}, \eqref{Eq:bvcoer2}, \eqref{Eq:Bbbound} and \eqref{Eq:Cbbound}. Finally, $H$-coercivity and $H$-boundedness of $a_0(\cdot,\cdot)$ imply generation of an analytic semigroup on $X$, see e.g. \cite{Banks1997}.
\end{proof}

To formulate the cascade system as an abstract boundary control system in the sense of \cite[Ch. 3.3]{Curtain1995}, we next define the related operators. In what follows $\mathbb{P}$ denotes the Leray projector as defined in \cite[Lemma 2.2]{Nguyen2015}, which is used to eliminate the pressure term while imposing incompressibility. Define the operators
\begin{align*}
	\A\begin{bmatrix}
	v \\ \theta \\ x_a \\ x_s
	\end{bmatrix} &= \begin{bmatrix}
		\mathbb{P}\bigg(\frac{1}{Re}\lapl{v} - (w_{ss} \cdot \nabla) v - (v \cdot \nabla)w_{ss} + \frac{Gr}{Re^2}\hat{e}_2  \theta \bigg) \\
		\frac{1}{RePr}\lapl{\theta} - w_{ss}\cdot \nabla \theta - v \cdot \nabla T_{ss} \\
		A_ax_a \\
		A_sx_s + B_sC_b[v, \ \theta]^T
	\end{bmatrix}:D(\A) \to X,  \\
	B_{\Gamma_u} &=\begin{bmatrix}
	b_v & 0 & 0 \\
	0 & b_{\theta_I} & 0 \\
	0 & 0 & b_{\theta_H}
	\end{bmatrix}:\mathbb{R}^{m_b} \to X_\Gamma, \\ B_{\Gamma_{u_d}} &= \begin{bmatrix}
	b_{dv} & 0 & 0 \\
	0 & b_{d\theta_I} & 0 \\
	0 & 0 & b_{d\theta_H}
	\end{bmatrix}:U_d \to X_\Gamma, \qquad
	B_\Gamma = \begin{bmatrix}
	B_{\Gamma_u} & B_{\Gamma_{u_d}}
	\end{bmatrix}.
\end{align*}
Operator $\A$ coincides with $A$ in $D(A)$ but has a larger domain due to relaxed boundary conditions within the disturbed parts of the boundary. That is, noting that Neumann trace of $H_b$ functions is in $(H^{-1/2}(\Gamma))^2 \times H^{-1/2}(\Gamma)$ and recalling the definition of $a_g(\cdot,\cdot)$ in \eqref{Eq:bilinag}, the domain is given by
\begin{align*}
	D(\A) &= \big\{ x \in H \big| \ \exists g_v \in (H^{-1/2}(\Gamma_I))^{2}, \ \exists g_{\theta_I} \in H^{-1/2}(\Gamma_I), \\& \qquad \exists g_{\theta_H} \in H^{-1/2}(\Gamma_H): 
	\forall \Psi \in H,
	 \Psi \to a_g(x,\Psi) \textrm{ is $X$-continuous} \big\}.
\end{align*}

Corresponding to the control and disturbance boundary conditions \eqref{Eq:boundarycondtv}-\eqref{Eq:boundarycondtthh}, we define the operator
\begin{equation*}
		\tilde{\B}_b\begin{bmatrix}
	v \\ \theta \\ p
	\end{bmatrix} = \begin{bmatrix}
	\big( \mathcal{T}(v,p) \cdot n + \alpha_vv \big)\big|_{\Gamma_I} \\
	\big(\frac{1}{RePr}\doo{\theta}{n} + \alpha_\theta \theta\big)\big|_{\Gamma_I} \\
	\big(\frac{1}{RePr}\doo{\theta}{n}\big)\big|_{\Gamma_H}
	\end{bmatrix}:D(\tilde{\B}_b) \subset X_b \times L^2(\Omega) \to X_\Gamma.
\end{equation*}
The pressure $p$ is uniquely defined by the velocity $v$, see \cite{Badra2012}, thus there exists operator $\B_b$ such that
\begin{equation*}
	\B_b\begin{bmatrix}
	v \\
	 \theta
	\end{bmatrix} = \tilde{\B}_b\begin{bmatrix}
	v \\ \theta \\ p
	\end{bmatrix} \qquad \forall \begin{bmatrix}
	v \\ \theta \\ p
	\end{bmatrix} \in D(\tilde{\B}_b), \qquad D(\B_b) = \big\{(v,\theta) \in D(\tilde{\B}_b)\big\}.
\end{equation*}
To match the state variable $x \in X$ of the cascade system, we finally define the operator
\begin{equation*}
		\B = \begin{bmatrix}
	\B_b & 0 & 0
	\end{bmatrix}: D(\B) = D(\B_b) \times \mathbb{R}^{n_a} \times \mathbb{R}^{n_s} \subset X \to X_\Gamma.
\end{equation*}

Since $A = \A|_{\N(\B-B_{\Gamma_u}[0, \ C_a, \ 0]^T)}$ generates an analytic semigroup on $X$ by Theorem \ref{thm:bilin} and $\B$ is onto $X_\Gamma$, cf. \cite{Burns2016}, by defining
\begin{equation}\label{Eq:B}
B = \begin{bmatrix}
0_{X_b} & B_a & 0_{\mathbb{R}^{n_s}}\end{bmatrix}^T \in \LL(U,X)
\end{equation}
we have that the cascade system \eqref{Eq:lintransbous}-\eqref{Eq:sensor} corresponds to the abstract boundary control system
\begin{subequations}\label{Eq:BCSA}
\begin{align}
\dot{x}(t) &= \A x(t) + Bu(t), \\
\B x(t) &= B_\Gamma \begin{bmatrix}
C_ax_a(t) \\ u_d(t)
\end{bmatrix}
\end{align}
on $X$ with the (boundary) input space $X_\Gamma$ in the sense of \cite[Ch. 3.3]{Curtain1995}. The system observation is given by
\begin{equation}\label{Eq:C}
	y(t) = Cx(t), \qquad C = [0_{X_b}, \ 0_{\mathbb{R}^{n_a}}, \ C_s]\in \LL(X,Y).
\end{equation}
\end{subequations}
Note that the control and observation operators of the boundary control system are bounded and the disturbance $B_{\Gamma_{u_d}}u_d$ is the only boundary input of the boundary control system. The boundary control system formulation of the cascade system also has the following equivalent state space formulation.

\begin{Proposition}
	The cascade system \eqref{Eq:lintransbous}-\eqref{Eq:sensor} can be formulated as
	\begin{subequations}\label{Eq:abslinsys}
		\begin{align}
		\dot{x}(t) &= Ax(t) + Bu(t) + B_du_d(t), \qquad x(0) = x_0 \in X, \\
		y(t) &= Cx(t) + D_du_d(t),
		\end{align}
	\end{subequations}
where the dynamics operator $A$ defined as in \eqref{Eq:opA} generates an analytic semigroup on the state space $X$ defined in \eqref{Eq:statespace} and the control operator $B$ together with the observation operator $C$ defined in \eqref{Eq:B} and \eqref{Eq:C} are bounded. Additionally, a change of the state variable $x$ can be applied such that also the resulting disturbance operator $B_d$ is bounded.
\end{Proposition}
\begin{proof}
	Existence of the state space formulation \eqref{Eq:abslinsys} follows from the boundary control system formulation as presented in \cite[Ch. 10]{TucsnakWeiss} and boundedness of $B$ and $C$ is apparent from their definitions. The change of variable $\tilde{x}_b=x_b-B_iB_{\Gamma_{u_d}}$, where $B_i$ is a right inverse of $\B_b$, used to homogenize the boundary conditions and obtain a bounded operator $B_d$ is presented in \cite[Ch. 3.3]{Curtain1995}. Note that the change of variable abuses smoothness of the disturbance signal $u_d$ given by \eqref{Eq:distsig} and introduces a bounded feedthrough operator $D_d$ into the system.
\end{proof}
\noindent We note that the state space formulation together with the fact that the operators $B,B_d,C$ and $D_d$ are bounded will later in this paper be utilized for implementing the output tracking controller. 
\begin{remark}
	The controller to be implemented will use no information on the disturbance related operators $B_d$ and $D_d$, thus we do not formulate the cascade system using the state variable $\tilde{x}_b$. However, one needs to verify that a presentation using bounded disturbance operators $B_d$ and $D_d$ exists.
\end{remark}

%% file: StabDet.tex
\subsection{Stabilizability and Detectability of the System}\label{Ssec:StabDet}
We will be using a controller including an observer, which means that we need to address both stabilizability and detectability properties of the cascade system \eqref{Eq:lintransbous}-\eqref{Eq:sensor}.
To begin with we note that, in addition to the cascade system, also the linearized translated Boussinesq equations \eqref{Eq:lintransbous} form an abstract boundary control system
\begin{align*}
\dot{x}_b(t) &= \A_bx_b(t), \\
\B_b x_b(t) &= B_\Gamma \begin{bmatrix}
u_b(t) \\ u_d(t)
\end{bmatrix}
\end{align*}
with the observation \eqref{Eq:BoussObs}. This can be verified by repeating the steps of Section \ref{ssec:abstractsys} without the actuator and sensor dynamics. Here
\begin{align*}
\A_b &= \A|_{D(\A_b)}, \\
D(\A_b) &= \big\{ x_b \in H_b \big| \ \exists g_v \in (H^{-1/2}(\Gamma_I))^2, \ \exists g_{\theta_I} \in H^{-1/2}(\Gamma_I), \\ &\qquad \exists g_{\theta_H} \in H^{-1/2}(\Gamma_H): \forall \varphi \in H_b, \ \varphi \to a_b(x,\varphi) \textrm{ is $X_b$-continuous} \big\}
\end{align*}
with the bilinear form $a_b(\cdot,\cdot)$ defined by
\begin{align*}
	a_b((v,\theta),(\psi,\phi))
	=& a_v(v,\psi) + a_\theta(\theta,\phi) + b_v(v,w_{ss},\psi) + b_v(w_{ss},v,\psi) \\ &+ b_\theta(w_{ss},\theta,\phi) + b_\theta(v,T_{ss},\phi)-b_0(\theta,\psi) - \langle g_v,\psi \rangle_{\Gamma_I}  \\
	&- \langle g_{\theta_I}, \phi \rangle_{\Gamma_I} - \langle g_{\theta_H},\phi \rangle_{\Gamma_H}.
\end{align*}
The associated generator of a strongly continuous semigroup is $A_b = \A_b|_{\N(\B_b)}$ and the associated control and disturbance operators for the abstract state space formulation are 
\begin{equation*}
B_b = (\A_b-A_{b_{-1}})B_iB_{\Gamma_u}, \qquad B_{bd} = (\A_b-A_{b_{-1}})B_iB_{\Gamma_{u_d}},
\end{equation*}
where $B_i$ is a right inverse of $\B_b$ and $A_{b_{-1}}$ is an extension of $A$, see \cite[Ch. 10]{TucsnakWeiss}. Now an alternative presentation for the operator $A$ initially defined in \eqref{Eq:opA} is given by
\begin{equation}\label{Eq:Aalt}
	A = \begin{bmatrix}
	A_b & B_bC_a & 0 \\
	0 & A_a & 0 \\
	B_sC_b & 0 & A_s
	\end{bmatrix},
\end{equation}
which is the form that we use for the stabilizability and detectability analysis.

Recall that by Theorem \ref{thm:bilin} $A$ generates an analytic semigroup on $X$. Additionally, by Theorem \ref{thm:bilin}, Lax--Milgram theorem and compactness of the embedding $H$ onto $X$, the resolvent of $A$ is compact on $X$, c.f. \cite{Ramaswamy2019}. Thus $A$ has a finite number of isolated eigenvalues on the closed right half plane $\overline{\mathbb{C}_0^+}$, each with finite multiplicity.
As such, stabilizability and detectability considerations of the cascade system with the bounded control operator $B$ and the bounded observation operator $C$ can be treated as controllability and observability problems of the finite-dimensional unstable part, see \cite[Ch. 5.2]{Curtain1995}. That is, the pair $(A,C)$ is exponentially detectable if and only if
\begin{align}\label{Eq:isdet}
\N(sI-A) \cap \N(C) = \{0\} \qquad \textrm{for all }s \in \overline{\mathbb{C}_0^+}.
\end{align}
Since exponential stabilizability and exponential detectability are dual concepts,
exponential stabilizability of the pair $(A,B)$ is then a matter of exponential detectability of the dual pair $(A^*,B^*)$.

Let $P_b(s)\coloneqq C_b(sI-A_b)^{-1}B_b$ for $s \in \rho(A_b)$ and the analogously defined $P_a$ and $P_s$ denote the transfer functions of systems \eqref{Eq:lintransbous}, \eqref{Eq:actuator} and \eqref{Eq:sensor}, respectively.

\begin{assumption}\label{ass:isdet} Assume that the following hold:
	\begin{enumerate}[(i)]
		\item The spectra $\sigma(A_b)$, $\sigma(A_a)$ and $\sigma(A_s)$ are pairwise disjoint on $\overline{\mathbb{C}_0^+}$.
		\item The pair $(A_s,C_s)$ is detectable.
		\item For every $\lambda \in \overline{\mathbb{C}_0^+}$, $\N(P_s(\lambda)C_b) \cap \N(\lambda I-A_b) = \{0\}$.
		\item For every $\lambda \in \overline{\mathbb{C}_0^+}$, $\N\big(P_s(\lambda)P_b(\lambda)C_a\big) \cap \N(\lambda I-A_a) = \{0\}$.
	\end{enumerate}
\end{assumption}
\noindent Note that by the results in \cite{Badra2014} the condition \eqref{Eq:isdet} holds for the system $(A_b,B_b,C_b)$ if and only if the system is exponentially detectable even if the control and observation are unbounded. Thus Assumption \ref{ass:isdet}.(iii),(iv) includes assumption of exponential detectability of $(A_b,C_b)$ and $(A_a,C_a)$. 

\begin{assumption}\label{ass:isstab} Assume that the following hold:
	\begin{enumerate}[(i)]
		\item The spectra $\sigma(A_b^*)$, $\sigma(A_a^*)$ and $\sigma(A_s^*)$ are pairwise disjoint on $\overline{\mathbb{C}_0^+}$.
		\item The pair $(A_a,B_a)$ is stabilizable.
		\item For every $\lambda \in \overline{\mathbb{C}_0^+}$, $\N(P_a(\overline{\lambda})^*B_b^*)\cap \N(\lambda I-A_b^*) = \{0\}$.
		\item For every $\lambda \in \overline{\mathbb{C}_0^+}$, $\N\big( P_a(\overline{\lambda})^*P_b(\overline{\lambda})^*B_s^* \big) \cap \N(\lambda I-A_s^*) = \{0\}$.
	\end{enumerate}
\end{assumption}
\noindent Again, exponential stabilizability of $(A_b,B_b)$ and $(A_s,B_s)$ is necessary for Assumption \ref{ass:isstab}.(iii),(iv) to hold.
\begin{Lemma}\label{lemma:isdet}
	If Assumption \ref{ass:isdet} holds, then the pair $(A,C)$ is exponentially detectable.
\end{Lemma}
\begin{proof}
	To check that \eqref{Eq:isdet} holds for the pair $(A,C)$, let $(x_b,x_a,x_s) \in \N(\lambda I-A)\cap \N(C)$, where $\lambda \in \overline{\mathbb{C}_0^+}$. Using \eqref{Eq:B}, \eqref{Eq:C} and \eqref{Eq:Aalt} we have
\begin{subequations}
	\begin{numcases}{}
		C_sx_s = 0, \label{D1} \\
		(\lambda I-A_a)x_a = 0, \label{D2}\\
		(\lambda I-A_b)x_b - B_bC_ax_a = 0, \label{D3} \\
		(\lambda I-A_s)x_s - B_sC_bx_b = 0. \label{D4}
	\end{numcases}
	\end{subequations}
	If $\lambda \in \rho(A_a)$, \eqref{D2} implies $x_a = 0$. If $\lambda \in \rho(A_a) \cap \rho(A_b)$, \eqref{D3} implies $x_b = 0$ and then Assumption \ref{ass:isdet}.(ii), \eqref{D1} and \eqref{D4} imply $x_s = 0$. If $\lambda \in \rho(A_a) \cap \sigma(A_b)$, then
	$x_s = (\lambda I - A_s)^{-1}B_sC_bx_b$ by \eqref{D4} and Assumption \ref{ass:isdet}.(i). By \eqref{D1} and Assumption \ref{ass:isdet}.(iii) we have $x_b=0$, which then implies $x_s = 0$.
	
	If $\lambda \in \sigma(A_a)$, Assumption \ref{ass:isdet}.(i) and \eqref{D3} imply $x_b = (\lambda I-A_b)^{-1}B_bC_ax_a$. Using Assumption \ref{ass:isdet}.(i) again with \eqref{D1} and \eqref{D4} then yields $0 = C_s(\lambda I-A_s)^{-1}B_sC_b(\lambda I-A_b)^{-1}B_bC_ax_a = P_s(\lambda)P_b(\lambda)C_ax_a$, thus $x_a=0$ by Assumption \ref{ass:isdet}.(iv) and \eqref{D2}. Now also $x_s=0$ and $x_b=0$ by e.g. \eqref{D3} and \eqref{D4}. As such, we have $\N(\lambda I - A) \cap \N(C) = \{0\}$ for any $\lambda \in \overline{\mathbb{C}_0^+}$, thus $(A,C)$ is exponentially detectable.
\end{proof}

\begin{Lemma}\label{lemma:isstab}
	If Assumption \ref{ass:isstab} holds, then the pair $(A,B)$ is exponentially stabilizable.
\end{Lemma}
\begin{proof}
	We check that \eqref{Eq:isdet} holds for the pair $(A^*,B^*)$. Let $(x_b,x_a,x_s) \in \N(\lambda I-A^*) \cap \N(B^*)$, where $\lambda \in \overline{\mathbb{C}_0^+}$. Using \eqref{Eq:B}, \eqref{Eq:C} and \eqref{Eq:Aalt} we have
	\begin{subequations}
		\begin{numcases}{}
			B_a^*x_a = 0, \label{S1} \\
			(\lambda I-A_s^*)x_s = 0, \label{S2} \\
			(\lambda I - A_b^*)x_b - C_b^*B_s^*x_s = 0, \label{S3} \\
			(\lambda I - A_a^*)x_a - C_a^*B_b^*x_b = 0. \label{S4}
		\end{numcases}
	\end{subequations}
	If $\lambda \in \rho(A^*_s)$, then \eqref{S2} immediately implies $x_s = 0$. If $\lambda \in \rho(A^*_s) \cap \rho(A_b^*)$, \eqref{S3} implies $x_b=0$, thus also $x_a=0$ by Assumption \ref{ass:isstab}.(ii), \eqref{S1} and \eqref{S4}. If $\lambda \in \rho(A_s^*) \cap \sigma(A_b^*)$, then \eqref{S4} and Assumption \ref{ass:isstab}.(i) imply $x_a = (\lambda I - A_a^*)^{-1}C_a^*B_b^*x_b$, thus $x_b=0$ by \eqref{S1} and Assumption \ref{ass:isstab}.(iii), from which $x_a=0$ follows.
	
	If $\lambda \in \sigma(A_s^*)$, then $x_b = (\lambda I - A_b^*)^{-1}C_b^*B_s^*x_s$ by \eqref{S3} and Assumption \ref{ass:isstab}.(i). Using Assumption \ref{ass:isstab}.(i) again with \eqref{S4} and \eqref{S1}, we get $0 = B^*_a(\lambda I - A_a^*)^{-1}C_a^*B_b^*(\lambda I -A_b^*)^{-1}C_b^*B_s^*x_s = P_a(\overline{\lambda})^*P_b(\overline{\lambda})^*B_s^*x_s$.
	Now $x_s = 0$ by Assumption \ref{ass:isstab}.(iv) and \eqref{S2}, and $x_a=0$, $x_b=0$ follow immediately, thus $\N(\lambda I - A^*) \cap \N(B^*) = \{0\}$ and $(A,B)$ is exponentially stabilizable.
\end{proof}

%% file: CLCD.tex
\section{Robust Output Regulation}\label{Sec:ROR}
The output tracking goal \eqref{Eq:desiredtracking} is in the case of abstract linear systems covered by the \emph{robust output regulation problem}. We start by coupling an error feedback controller with the cascade system \eqref{Eq:abslinsys}. The resulting system is called the \emph{closed-loop system}. We then present the robust output regulation problem, which describes requirements for choosing the controller operators. Finally, we design an error feedback controller, introduced in \cite{Paunonen2019}, to solve the robust output regulation problem for the room model.

A general error feedback controller on a Hilbert space $Z$ is given by
\begin{subequations}\label{Eq:errorfbcont}
	\begin{align}
	\dot{z}(t) &= \mathcal{G}_1z(t) + \mathcal{G}_2e(t),  \qquad z(0) = z_0 \in Z, \\
	u(t) &= Kz(t),
	\end{align}	
\end{subequations}
where $\mathcal{G}_1$ generates a strongly continuous semigroup on $Z$, $\mathcal{G}_2 \in \LL(Y,Z)$, $K \in \LL(Z,U)$ and $e(t)=y(t)-y_{ref}(t)$ is the regulation error.
Coupling the controller with the cascade system \eqref{Eq:abslinsys} yields the closed-loop system, see  \cite{Hamalainen2010,Paunonen2010},
\begin{align*}
\dot{x}_e(t) &= A_ex_e(t) + B_ew_{ext}(t), \qquad x_e(0) = x_{e0}, \\
e(t) &= C_ex_e(t) + D_ew_{ext}(t)
\end{align*}
on the Hilbert space $X_e \coloneqq X \times Z$ with the state $x_{e} = [x, \ z]^T$.
Here $w_{ext} = [u_d, \ y_{ref}]^T$,
\begin{alignat*}{2}
A_e &= \begin{bmatrix}
A & BK \\
\mathcal{G}_1C & \mathcal{G}_1
\end{bmatrix}, \qquad &&B_e = \begin{bmatrix}
B_d & 0 \\
\mathcal{G}_2 D_d & - \mathcal{G}_2
\end{bmatrix}, \\
C_e &= \begin{bmatrix}
C, & 0
\end{bmatrix}, &&D_e = \begin{bmatrix}
D_d, & -I
\end{bmatrix}.
\end{alignat*}

\noindent\textbf{The Robust Output Regulation Problem.}
Design a controller \linebreak[4] $(\mathcal{G}_1,\mathcal{G}_2,K)$ such that the following hold:
\begin{enumerate}[(I)]
	\item The closed-loop semigroup is exponentially stable.
	\item  There exist $M_r, \omega_r > 0$ such that for all initial states $x_{e0} \in X_e$ of the closed-loop system and for all reference signals $y_{ref}$ in \eqref{Eq:refsig} and disturbance signals $u_d$ in \eqref{Eq:distsig}
	\begin{align}\label{Eq:regerr}
	\Vert y(t)-y_{ref}(t)\Vert \le M_re^{-\omega_r t}(\Vert x_{e0}\Vert + \Vert \Lambda\Vert),
	\end{align}
	where $\Lambda$ is a vector consisting of the coefficients $a_i(t)$, $b_i(t)$, $c_i(t)$ and $d_i(t)$ of $y_{ref}$ and $u_d$.
	\item If $A,B,B_{d},C,D_{d}$ in \eqref{Eq:abslinsys} are perturbed to $\tilde{A},\tilde{B},\tilde{B}_{d},\tilde{C},\tilde{D}_{d}$ in such a way that the closed-loop system remains exponentially stable, then for all $x_{e0} \in X_e$ and for all signals of the form \eqref{Eq:refsig}, \eqref{Eq:distsig} the regulation error satisfies \eqref{Eq:regerr} for some $\tilde{M}_r,\tilde{\omega}_r > 0$.
\end{enumerate}

By the internal model principle, a controller solves the robust output regulation problem precisely when it includes a suitable \emph{internal model} of the reference and disturbance signals and the closed-loop system is exponentially stable \cite{Paunonen2010}. The following lemma, i.e. $(A,B,C)$ having no transmission zeros at the relevant frequencies, is a standard necessary property for solvability of the problem.
\begin{assumption}\label{ass:invariant}
	None of the systems $(A_b,B_b,C_b)$, $(A_a,B_a,C_a)$ and \linebreak[4]$(A_s,B_s,C_s)$ has transmission zeros at the frequencies $\{ i\omega_k\}_{k=0}^{q_s}$, i.e. for a bounded stabilizing feedback operator $K_b$, the transfer function $P_{b,K_b}(i\omega_k)=(C_b)(i\omega_kI-A_b-B_bK_b)^{-1}B_b$ is surjective for $k = 0,1,\dots,q_s$ and the equivalent result holds for $P_{a,K_a}$ and $P_{s,K_s}$.	
\end{assumption}
\begin{Lemma}\label{lemma:invariant}
 Given Assumption \ref{ass:invariant}, the cascade system $(A,B,C)$ has no transmission zeros at the frequencies $\{ i\omega_k\}_{k=0}^{q_s}$.
\end{Lemma}
\noindent The proof follows immediately since transfer function of the cascade system is the product of the transfer functions of the three subsystems.

The particular controller design we propose for the cascade system is the one in \cite[Sec. III.A]{Paunonen2019}, see also \cite{Phan2019} for its boundary control system implementation.

\noindent \textbf{The Observer-Based Finite-Dimensional Controller} is given by
\begin{subequations}\label{Eq:obscont}
	\begin{align}
	\dot{z}_1(t) &= G_1z_1(t) + G_2e(t), \\
	\dot{z}_2(t) &= (A_L^r + B_L^r K_2^r) z_2(t) + B_L^r K_1^N z_1(t) - L^re(t), \\
	u(t) &= K_1^N z_1(t) + K_2^r z_2(t), 
	\end{align}
\end{subequations}
and is of the form \eqref{Eq:errorfbcont} with $z(t) := [z_1(t), \ z_2(t)]^T \in Z :=Z_{im} \times \mathbb{C}^r$,
\begin{align*}
\mathcal{G}_1 = \begin{bmatrix}
G_1 & 0 \\ B_L^rK_1^N & A_L^r + B_L^rK_2^r
\end{bmatrix}, \qquad \mathcal{G}_2 = \begin{bmatrix}
G_2 \\ -L^r
\end{bmatrix}, \qquad K = \begin{bmatrix}
K_1^N & K_2^r
\end{bmatrix}.
\end{align*}
For the cascade system \eqref{Eq:lintransbous}-\eqref{Eq:sensor}, the operators in \eqref{Eq:obscont} are chosen according to the following algorithm.

\noindent\textbf{\RNum{1} The Internal Model:}
\newline Choose $Z_{im} = Y^{n_0} \times Y^{2n_1} \times \ldots \times Y^{2n_{q_s}}$, where $n_i-1$, $i \in \{0,1,...,q_s\}$ is the highest order polynomial coefficient of the corresponding frequency $\omega_i$ in \eqref{Eq:refsig}. Set $G_1 = \textrm{diag}\big(J_0^{Y}, \ldots J_{q_s}^{Y}\big) \in \LL(Z_{im})$ and $G_2 = \big(G_2^k\big)_{k=0}^{q_s} \in \LL(Y,Z_{im})$, where
\begin{align*}
J_0^{Y} = \begin{bmatrix}
0_p & I_p & & \\
& 0_p & \ddots & \\
& & \ddots & I_p \\
& & & 0_p
\end{bmatrix}, \qquad G_2^0 = \begin{bmatrix}
0_p \\ \vdots \\ 0_p \\ I_p
\end{bmatrix}
\end{align*}
and for $k=1 \dots q_s$
\begin{align*}
J_k^{Y} = \begin{bmatrix}
\Omega_k & I_{2p} & & \\
& \Omega_k & \ddots & \\
& & \ddots & I_{2p} \\
& & & \Omega_k
\end{bmatrix}, \qquad G_2^k = \begin{bmatrix}
0_{2p} \\ \vdots \\ 0_{2p} \\ I_p \\ 0_p
\end{bmatrix}, \qquad \Omega_k = \begin{bmatrix}
0_p & \omega_kI_p \\
-\omega_kI_p & 0_p
\end{bmatrix}.
\end{align*}

\noindent \textbf{\RNum{2} Plant Approximation and Stabilization:}
\newline For a sufficiently large $N \in \mathbb{N}$, discretize the operators $A_b$, $B_b$ and $C_b$ to obtain the finite-dimensional approximative system $(A_b^N,B_b^N,C_b^N)$ on $H_b^N$. The chosen approximation method should satisfy the following assumption.
\begin{assumption}\label{ass:approximation}
	The finite-dimensional approximating subspaces $H_b^N$ of $H_b$ have the following property: For any $x_b \in H_b$ there exists a sequence $x_b^N \in H_b^n$ such that 
	\begin{equation}\label{Eq:approxprop}
		\Vert x_b^N-x_b \Vert_{H_b} \to 0 \textrm{ 
			as } N \to \infty.
	\end{equation}
\end{assumption}
\noindent More specifically, approximation of the coupled-system should have the property equivalent to
\eqref{Eq:approxprop}, and Assumption \ref{ass:approximation} implies the approximation property for the cascade system.
\begin{Lemma}\label{lemma:approximation}
	Let $H^N \coloneqq H_b^N \times \mathbb{R}^{n_a} \times \mathbb{R}^{n_s}$ denote the finite-dimensional approximating subspaces of $H$. Given Assumption \ref{ass:approximation},
	there exists a sequence $x^N \in H^N$ such that $\Vert x^N-x\Vert_{H} \to 0$ as $N \to \infty$.
\end{Lemma}
\noindent The proof follows immediately by choosing $x^N = [x_b^N, \ x_a, \ x_s]^T$, where $x_b^N$ such that Assumption \ref{ass:approximation} holds.

Denote the cascade system approximation on $H^N$ by $(A^N,B^N,C^N)$.
Let $\alpha_1,\alpha_2 \ge 0$ and choose $Q_1 \in \LL(U_0,X)$, $Q_2 \in \LL(X,Y_0)$ such that $(A + \alpha_1I,Q_1,C)$ and $(A + \alpha_2I,B,Q_2)$ are exponentially stabilizable and detectable, where $U_0$ and $Y_0$ are Hilbert spaces. Denote by $Q_1^N$ and $Q_2^N$ the approximations of $Q_1$ and $Q_2$ on $H^N$. Choose $0 < R_1 \in \LL(Y)$ and $0 < R_2 \in \LL(U)$, and choose $Q_0 \in \LL(Z_{im},\mathbb{C}^{p_0})$ such that $(Q_0,G_1)$ is observable. Let
\begin{equation*}
A_c^N = \begin{bmatrix}
G_1 &  G_2C^N \\
0 & A^N
\end{bmatrix}, \qquad B_c^N = \begin{bmatrix}
0 \\ B^N
\end{bmatrix}, \qquad Q_c^N = \begin{bmatrix}
Q_0 & 0 \\ 0 & Q_2^N
\end{bmatrix}
\end{equation*} 
and solve the finite-dimensional Riccati equations
\begin{align*}
&(A^N+ \alpha_1I)\Sigma_N + \Sigma_N(A^N+\alpha_1I)^* - \Sigma_N(C^N)^*R_1^{-1}C^N\Sigma_N = -Q_1^N(Q_1^N)^*, \\
&(A_c^N + \alpha_2I)^*\Pi_N + \Pi_N (A_c^N+\alpha_2I) - \Pi_NB_c^NR_2^{-1}(B_c^N)^*\Pi_N = -(Q_c^N)^*Q_c^N.
\end{align*}
Finally, define $L^N = -\Sigma_NC^NR_1^{-1} \in \LL(Y,H^N)$ and $K^N \coloneqq [K_1^N, \ K_2^N] = -R_2^{-1}(B_c^N)^*\Pi_N \in \LL(Z_{im} \times H^N,U)$.

\noindent \textbf{\RNum{3} Model Reduction:}
\newline For a sufficiently large $ r \le N$, apply balanced truncation, see \cite[Sec. II-B]{Paunonen2019} and the references therein, on the stable system
\begin{equation*}
(A^N+L^NC^N,[B^N,L^N],K_2^N)
\end{equation*}
to obtain the reduced order system
\begin{equation*}
(A_L^r,[B_L^r, \ L^r],K_2^r).
\end{equation*}

By \cite[Thm. III.1]{Paunonen2019}, the Observer-Based Finite-Dimensional Controller solves the robust output regulation problem for a class of systems including the cascade system \eqref{Eq:lintransbous}-\eqref{Eq:sensor}. Thus the following holds for robust output tracking of the linearized translated Boussinesq equations \eqref{Eq:lintransbous}, \eqref{Eq:boussobs} with actuator dynamics \eqref{Eq:actuator} and sensor dynamics \eqref{Eq:sensor}.
\begin{Theorem}	
	Assume that the control and disturbance shape functions $[b_{v}, \ b_{dv}] \in (L^2(\Gamma_I))^i$, $[b_{\theta_I}, \ b_{d\theta_I}] \in (L^2(\Gamma_I))^j$ and $[b_{\theta_H}, \ b_{d\theta_H}] \in (L^2(\Gamma_H))^k$, where $i$, $j$ and $k$ equal the sum of the number of control and disturbance inputs for the inlet velocity, the inlet temperature and the heating strip temperature, respectively. Assume that the observation satisfies Assumption \ref{ass:admobs} and the cascade system \eqref{Eq:lintransbous}-\eqref{Eq:sensor} satisfies Assumptions \ref{ass:isdet}, \ref{ass:isstab} and \ref{ass:invariant}. Then the finite-dimensional low-order controller \eqref{Eq:obscont} solves the robust output regulation problem, thus achieving the output convergence \eqref{Eq:desiredtracking}, provided that Assumption \ref{ass:approximation} holds for the approximation method used in the controller design and the orders of approximation $N$ and $r \le N$ are large enough.	
\end{Theorem}

%% file: Example.tex
\section{Output Tracking Example for the Room Model}\label{Sec:example}
We consider robust output tracking of the linearized translated Boussinesq equations \eqref{Eq:lintransbous} with the actuator dynamics \eqref{Eq:actuator} and the sensor dynamics \eqref{Eq:sensor} in the rectangular room $\Omega = [0,1] \times [0,1]$ with the inlet, the outlet and the heating strip given by
\begin{align*}
\Gamma_I &= \bigg\{\xi_1 = 0, \ \frac{5}{8} \le \xi_2 \le \frac{7}{8} \bigg\}, \quad \Gamma_O = \bigg\{ \xi_1 = 1, \ \frac{1}{8} \le \xi_2 \le \frac{1}{2} \bigg\}, \\
\Gamma_H &= \bigg\{\frac{3}{8} \le \xi_1 \le \frac{5}{8}, \ \xi_2 = 0  \bigg\},
\end{align*}
which roughly corresponds to Figure \ref{fig:Room}. 
Physical parameters for the Boussinesq equations are chosen as $Re = 100$, $Gr = \frac{Re^2}{0.9}$ and $Pr = 0.7$.
There are three control inputs on the fluid; one on each of $v_{\xi_1}$ and $\theta$ within the inlet and one on $\theta$ within the heating strip, with coefficients $\alpha_v=\alpha_\theta=1$ indicating Robin boundary conditions within the inlet. Additionally, we assume that a single disturbance signal acts within the inlet on $v_{\xi_2}$. Now $u_b(t) = [u_v(t), \ u_{\theta_I}(t), \ u_{\theta_H}(t)]^T\in \mathbb{R}^3$ and $u_{d}(t) = u_{dv}(t) \in \mathbb{R}$. The control and disturbance shapes are given by
\begin{align*}
b_{v}(\xi_2) &= \begin{bmatrix}
\exp\bigg( \frac{-0.00004}{((5/8-\xi_2)(7/8-\xi_2))^2} \bigg) \bigg|_{\Gamma_I} , & 0
\end{bmatrix}^T, \\
b_{\theta_I}(\xi_2) &=
\exp\bigg( \frac{-0.00002}{((5/8-\xi_2)(7/8-\xi_2))^2} \bigg)\bigg|_{\Gamma_I}, \\
b_{\theta_H}(\xi_1) &=  
\exp\bigg( \frac{-0.00001}{((3/8-\xi_1)(5/8-\xi_1))^2} \bigg)\bigg|_{\Gamma_H}, \\
b_{dv}(\xi_2) &=\begin{bmatrix}
0, & \exp\bigg( \frac{-0.0003}{((5/8-\xi_2)(7/8-\xi_2))^2} \bigg)\bigg|_{\Gamma_I}
\end{bmatrix}^T
\end{align*}
with the non-zero components depicted in Figure \ref{fig:profiles}.
\begin{figure*}[h]
	\centering
	\begin{subfigure}[t!]{0.48\textwidth}
		\centering
		\includegraphics[width=\textwidth]{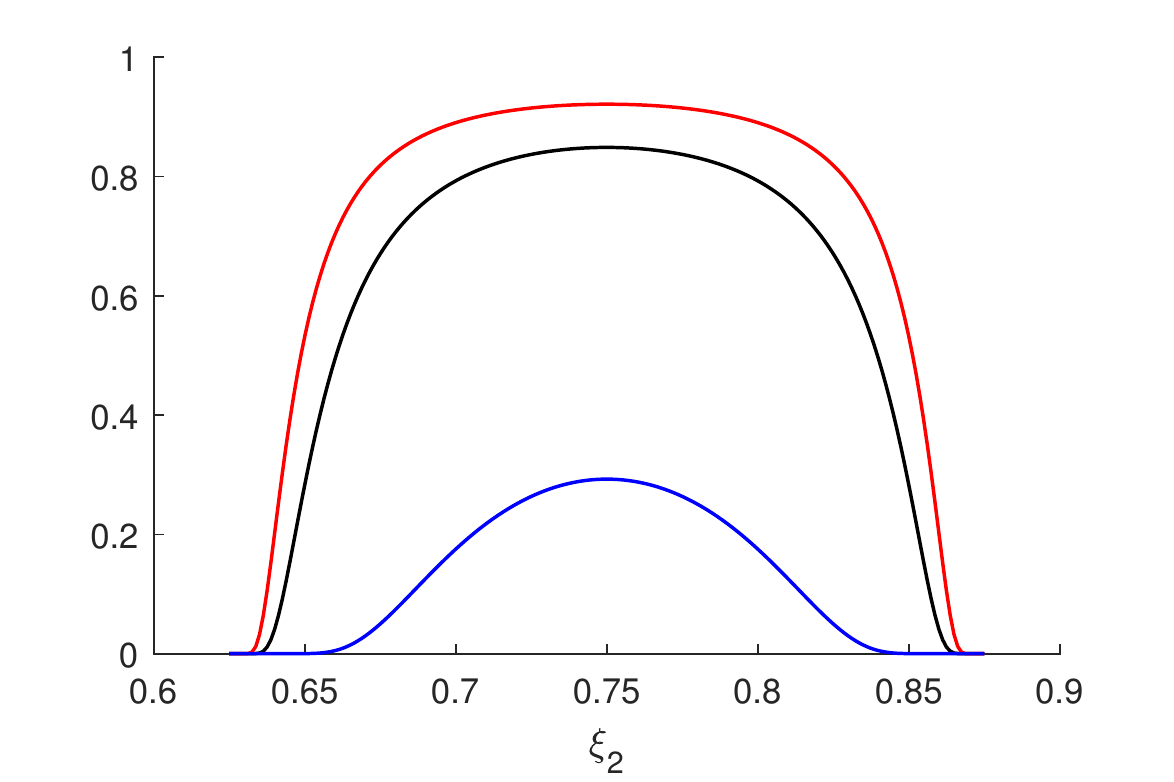}
	\end{subfigure}
	\hfill
	\begin{subfigure}[t!]{0.48\textwidth}  
		\centering 
		\includegraphics[width=\textwidth]{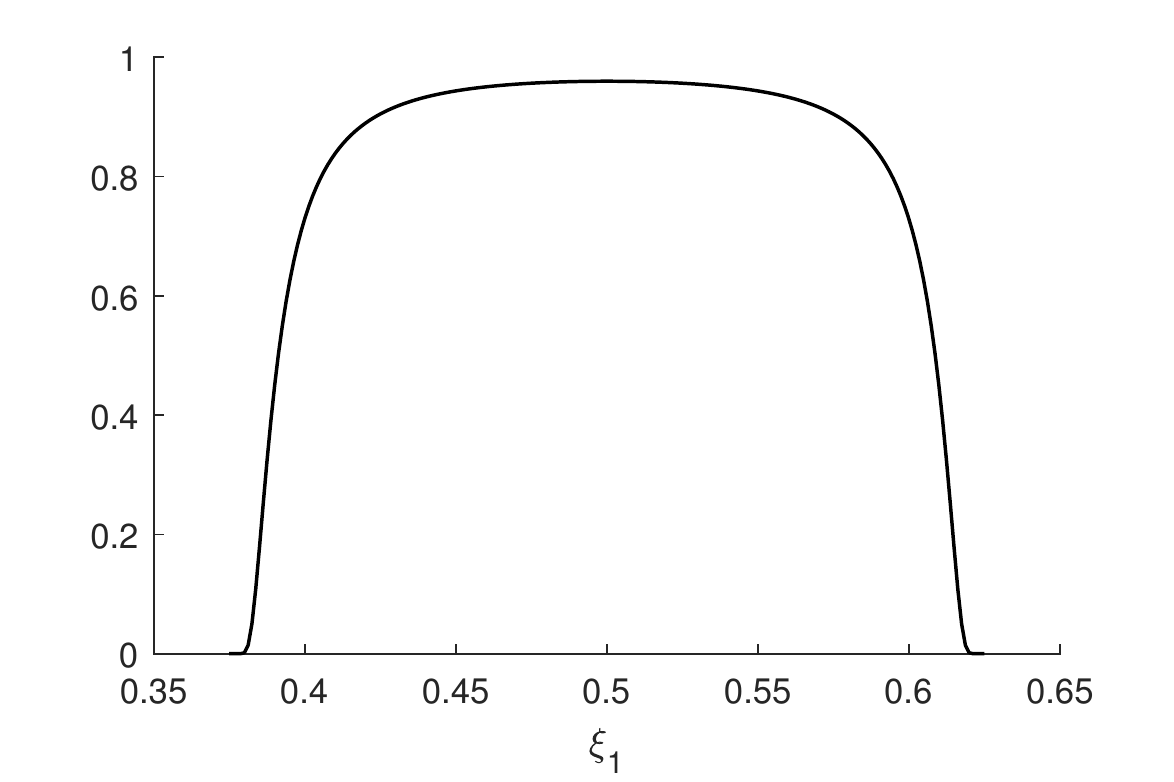}
	\end{subfigure}
	\captionsetup{justification=centering}
	\caption{Control and disturbance shape functions. On the left $b_{v}$ (black), $b_{\theta_I}$ (red) and $b_{dv}$ (blue), and on the right $b_{\theta_H}$.}
	\label{fig:profiles}
\end{figure*}

We consider three observations on the linearized Boussinesq equations \eqref{Eq:lintransbous}. These are given by
\begin{align*}
y_{b1}(t) &= \frac{1}{|\Omega_\theta|}\int_{\Omega_\theta}\theta(\xi,t)d\xi, \qquad y_{b2}(t) = \frac{1}{|\Gamma_O|}\int_{\Gamma_O} \theta(\xi,t)d\xi_2, \\
y_{b3}(t) &= \frac{1}{|\Omega_v|}\int_{\Omega_v} v_{\xi_1}(\xi,t)d\xi,
\end{align*}
where $\Omega_\theta = [\frac{1}{8}, \ \frac{2}{8}] \times [\frac{5}{8}, \ \frac{6}{8}]$ and $\Omega_v = [\frac{3}{8}, \ \frac{4}{8}] \times [\frac{2}{8}, \ \frac{3}{8}]$, and we set $y_b = [y_{b1},\ y_{b2}, \ y_{b3}]^T$. The actuator \eqref{Eq:actuator} and the sensor \eqref{Eq:sensor} are characterized by the simple choices
\begin{equation}\label{Eq:ASpicks}
	A_a = A_s = -I_3, \qquad B_a = C_a = B_s = C_s = I_3,
\end{equation}
and the reference signals to be tracked are
\begin{align*}
y_{ref1}(t) &= -1 + \sin(t) + 0.3\cos(2t), \qquad y_{ref2}(t) = 0.5\cos(0.5t), \\
y_{ref3}(t) &= 1+ 0.5\sin(2t),
\end{align*}
respectively. Finally, the disturbance signal is given by $u_{d}(t) = 2\sin(0.5t)$.

For the simulations, we use a uniform triangulation of $\Omega$ together with Taylor-Hood finite element spatial discretization for the Navier--Stokes part of the Boussinesq equations and quadratic elements with the same triangulation for the advection--diffusion part. The incompressibility condition $\nabla \cdot v = 0$ is relaxed by using a penalty method to decouple the pressure term from the velocity, see e.g. \cite[Ch. 5.2]{Gunzburger1989} or \cite{He2018}, with the penalty parameter $\epsilon_p = 10^{-5}$. To approximate the infinite-dimensional system $(A_b,B_b,C_b)$, we use the mesh size $h_{inf}=1/24$, which results in approximation order $N_{inf} = 6728$ for the system $(A,B,C)$ after accounting for the boundary conditions.

We use Newton's method to calculate a steady state solution for the Boussinesq equations \eqref{Eq:Bouss} subject to
\begin{align*}
f_T(\xi) &= 5\sin(2\pi \xi_1)\cos(2 \pi \xi_2), \\
f_w(\xi) &= 4\begin{bmatrix}
\sin(2\pi\xi_1)\cos(2\pi\xi_2), & -\cos(2\pi\xi_1)\sin(2\pi\xi_2)
\end{bmatrix}^T.
\end{align*}
The steady state solution may be, and according to numerical tests is, non-unique, and we choose the steady state $(w_{ss},T_{ss})$ corresponding to the initial guess given by the steady state solution $(w_{i},T_{i})$ of \eqref{Eq:Bouss} subject to
\begin{align*}
f_{T_i}(\xi) &= 4\sin(2\pi\xi_1)\cos(2\pi\xi_2), \\
f_{w_i}(\xi) &= 2\begin{bmatrix}
\sin(2\pi\xi_1)\cos(2\pi\xi_2), & -\cos(2\pi\xi_1)\sin(2\pi\xi_2)
\end{bmatrix}^T.
\end{align*}
The steady state solution is depicted in Figure \ref{fig:steadys}.
\begin{figure*}[h]
	\centering
	\begin{subfigure}[t]{0.48\textwidth}
		\centering
		\includegraphics[width=\textwidth]{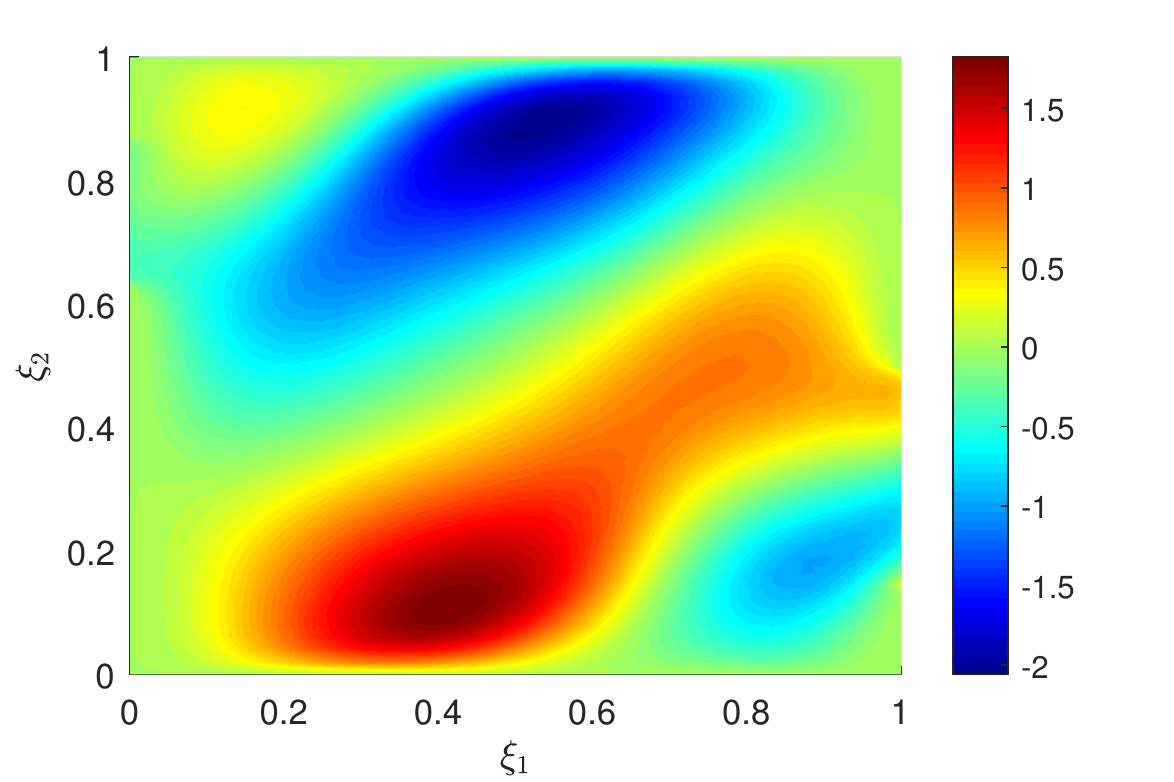}
		\captionsetup{justification=centering,font=normalsize}
		\caption[]%
		{The steady state velocity $w_{ss_1}$}    
		\label{fig:steadysxvel}
	\end{subfigure}
	\hfill
	\begin{subfigure}[t]{0.48\textwidth}  
		\centering 
		\includegraphics[width=\textwidth]{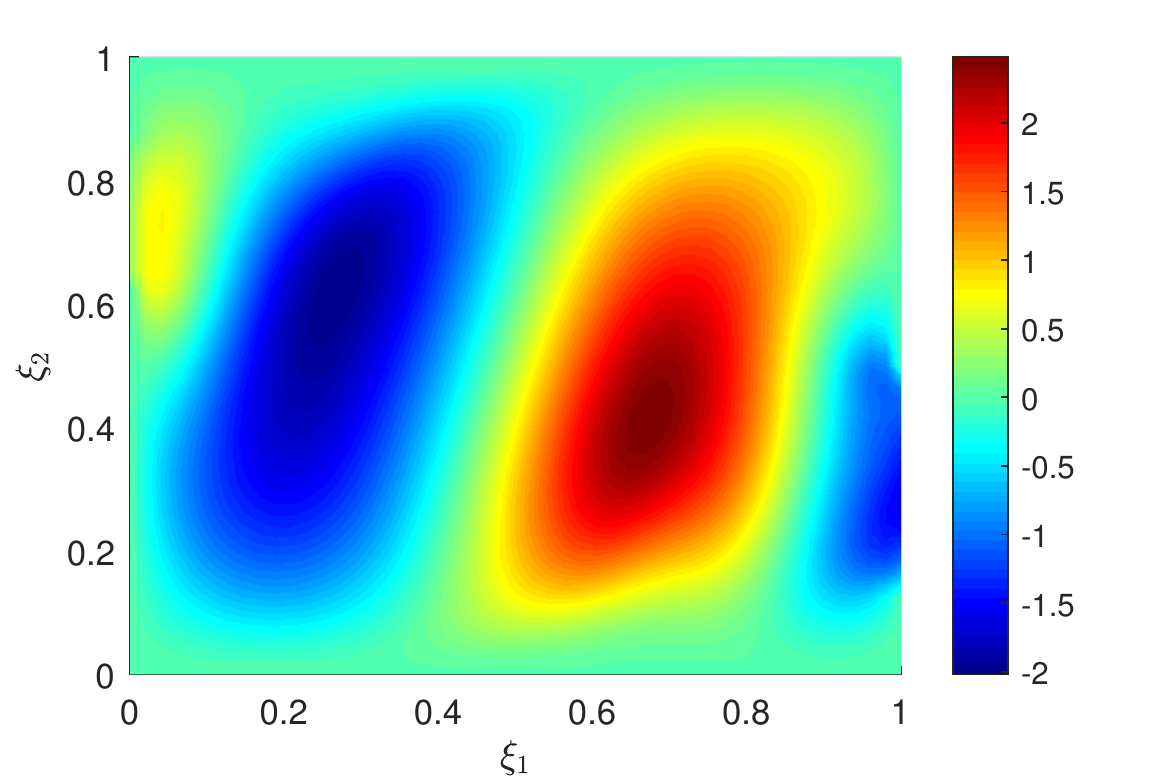}
		\captionsetup{justification=centering,font=normalsize}
		\caption[]%
		{ The steady state velocity $w_{ss_2}$}    
		\label{fig:steadysyvel}
	\end{subfigure}
	\vskip\baselineskip
	\begin{subfigure}[t]{0.48\textwidth}   
		\centering 
		\includegraphics[width=\textwidth]{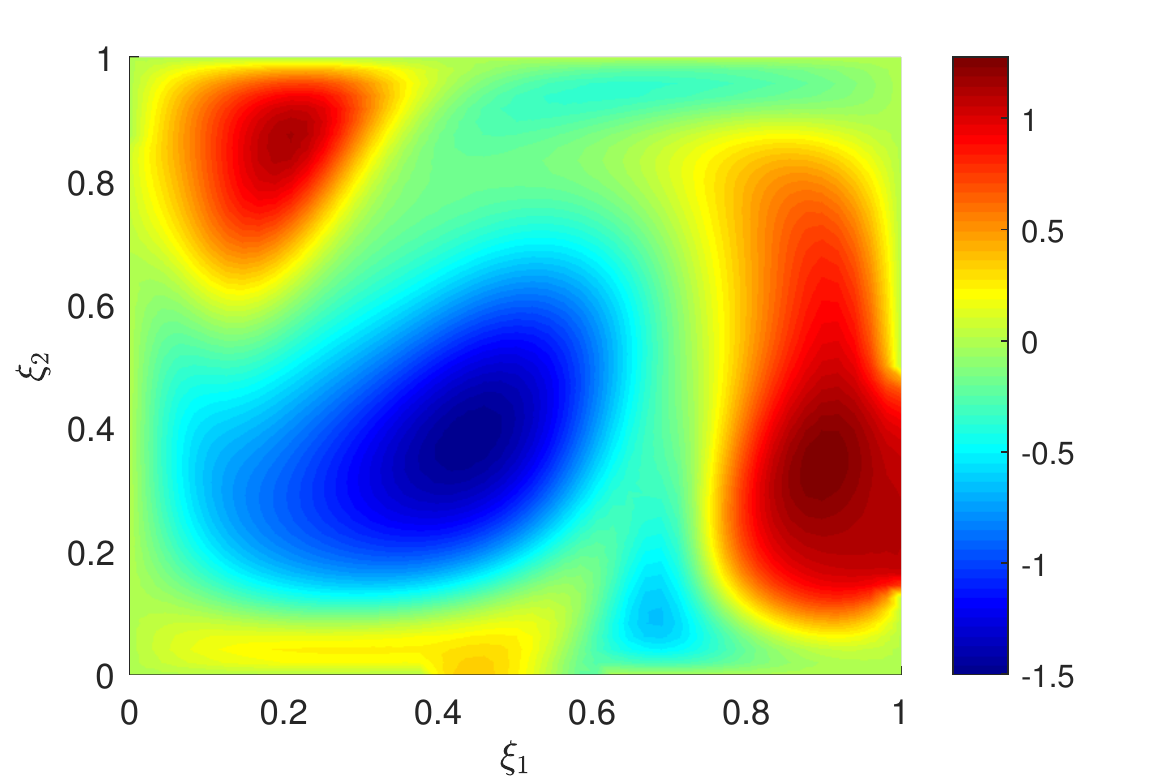}
		\captionsetup{justification=centering,font=normalsize}
		\caption[]%
		{ The steady state temperature $T_{ss}$}    
		\label{fig:steadysth}
	\end{subfigure}
	\quad
	\begin{subfigure}[t]{0.48\textwidth}   
		\centering 
		\includegraphics[width=\textwidth]{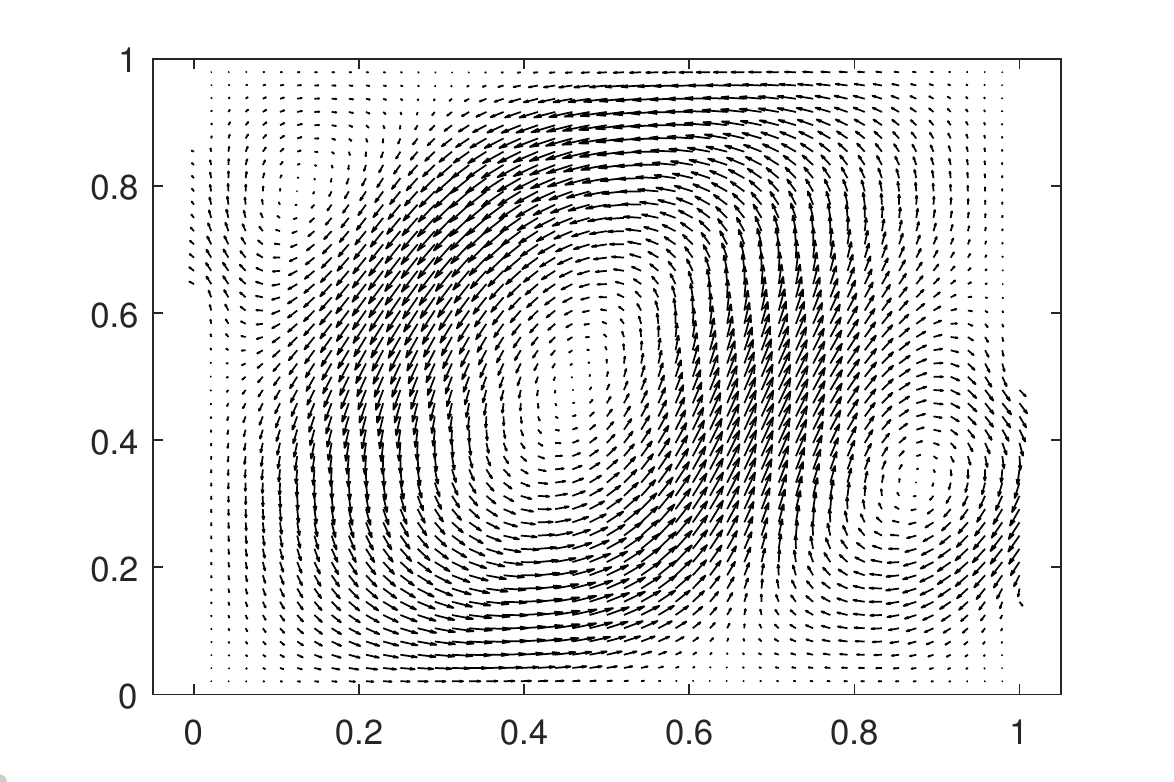}
		\captionsetup{justification=centering,font=normalsize}
		\caption[]%
		{ The  steady state velocity field $w_{ss}$}    
		\label{fig:steadysqvel}
	\end{subfigure}
	\captionsetup{justification=centering,font=normalsize}
	\caption[]
	{ A steady state solution $(w_{ss_1},w_{ss_2},T_{ss})$ of the Boussinesq equations \eqref{Eq:Bouss}} 
	\label{fig:steadys}
\end{figure*}
We observe numerically that for the calculated steady state solution $A_b$, thus also $A$ due to the block triangular structure if rearranged according to the state $(x_a,  x_b ,  x_s)$, has a single pair of unstable eigenvalues
\begin{equation*}
\lambda_{\pm} \approx 0.0621 \pm 0.4908i.
\end{equation*}
Exponential stabilizability, exponential detectability and Assumption \ref{ass:invariant} are checked numerically for the system $(A_b,B_b,C_b)$, and for the choice \eqref{Eq:ASpicks} they are transferred to the system $(A,B,C)$ by Lemmata \ref{lemma:isdet}, \ref{lemma:isstab} and \ref{lemma:invariant}.

For the controller construction, we use a coarser linearized Boussinesq equations approximation with the mesh size $h = 1/16$. Using the penalty method would introduce additional modeling error not compatible with Assumption \ref{ass:approximation}, so we base the plant approximation $(A^N,B^N,C^N)$ on the discretized Leray projector instead, see \cite{Heikenschloss2008,Benner2015,Bansch2015}. The discretized Leray projector is used merely as a theoretical tool, and the Riccati equations together with the model-reduced operators $A_L^r,B_L^r,K_2^r,L^r$ are instead solved using the differential-algebraic equation form
of the Taylor-Hood-quadratic discretized cascade system. More specifically, we obtain solutions of the Riccati equations using the Generalized Low-Rank Cholesky Factor Newton Method, see Algorithm 2 in \cite{Bansch2015}, with the initial stabilizing feedback solved by the Matlab's \texttt{icare} function using the penalized Taylor-Hood-quadratic discretization and the ADI shift parameters solved using the \texttt{LYAPACK} toolbox for Matlab. Finally, \texttt{balred} function of Matlab is utilized for the order reduction. Parameter's of the Riccati equations are chosen as
\begin{align*}
\alpha_1 = 0.3, \quad \alpha_2 = 0.2, \quad R_1=R_2=I_3, \quad Q_1Q_1^* = I_{X}, \quad Q_c^*Q_c = I_{ Z_{im} \times X}.
\end{align*}
To track $y_{ref} = [y_{ref1}, \ y_{ref2}, \ y_{ref3}]^T$ while rejecting $u_{d}$, the internal model of the controller has $4$ frequencies. Due to $y_{ref}$ and $u_d$ having constant coefficients, $\dim Z_{im} = 3+3 \cdot 3 \cdot 2=21$, thus us using the reduction order $r=20$ results in the controller order $\dim Z = \dim Z_{im} + r = 41$.

For the simulations, we choose as the initial state of the closed-loop system
\begin{align*}
\begin{bmatrix}
v_0, & \theta_0, & x_{a0}, & x_{s0}, & z_0
\end{bmatrix}^T &= \begin{bmatrix}
w_{i}-w_{ss}, & T_{i}-T_{ss}, & 0, & 0, & 0
\end{bmatrix}^T \\
&\quad \in X_v \times L^2(\Omega) \times \mathbb{R}^{3} \times \mathbb{R}^{3} \times Z
\end{align*}
and the state components $v_0$ and $\theta_0$ are shown in Figure \ref{fig:is}.
\begin{figure*}[h!]
	\centering
	\begin{subfigure}[t]{0.48\textwidth}
		\centering
		\includegraphics[width=\textwidth]{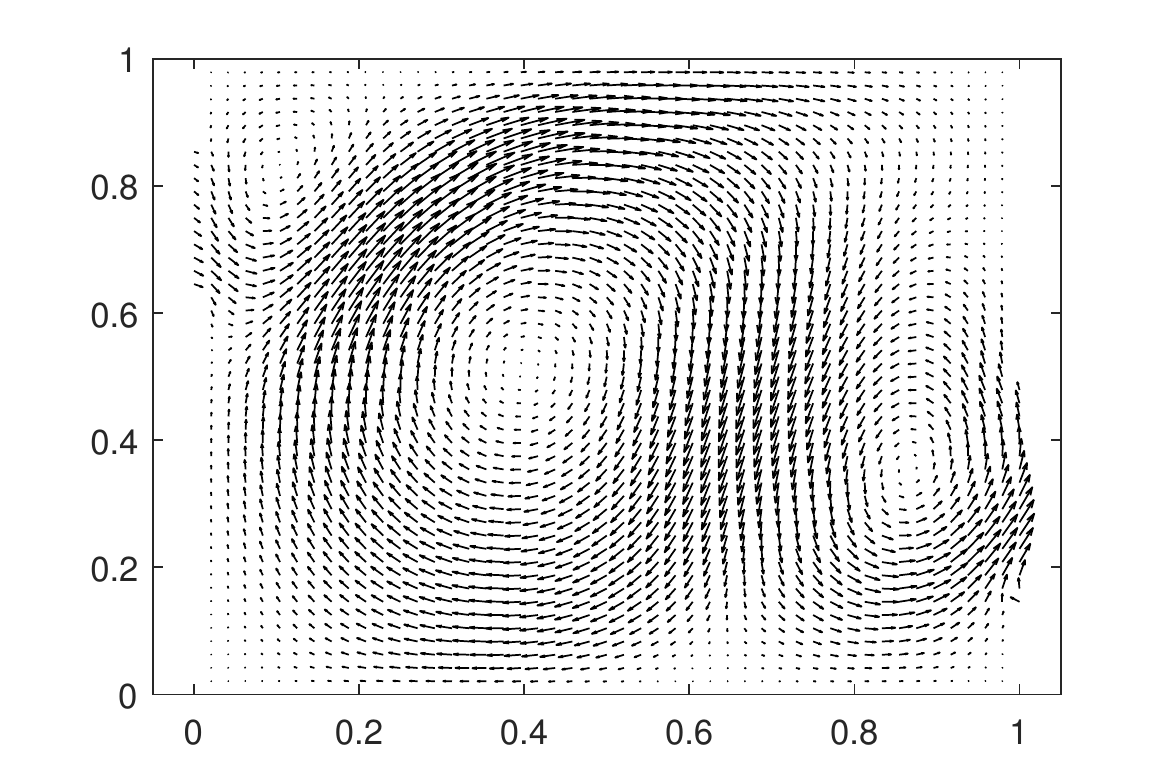}
		\captionsetup{justification=centering,font=normalsize}
		\caption[]%
		{$v_0$}    
		\label{fig:isvel}
	\end{subfigure}
	\hfill
	\begin{subfigure}[t]{0.48\textwidth}  
		\centering 
		\includegraphics[width=\textwidth]{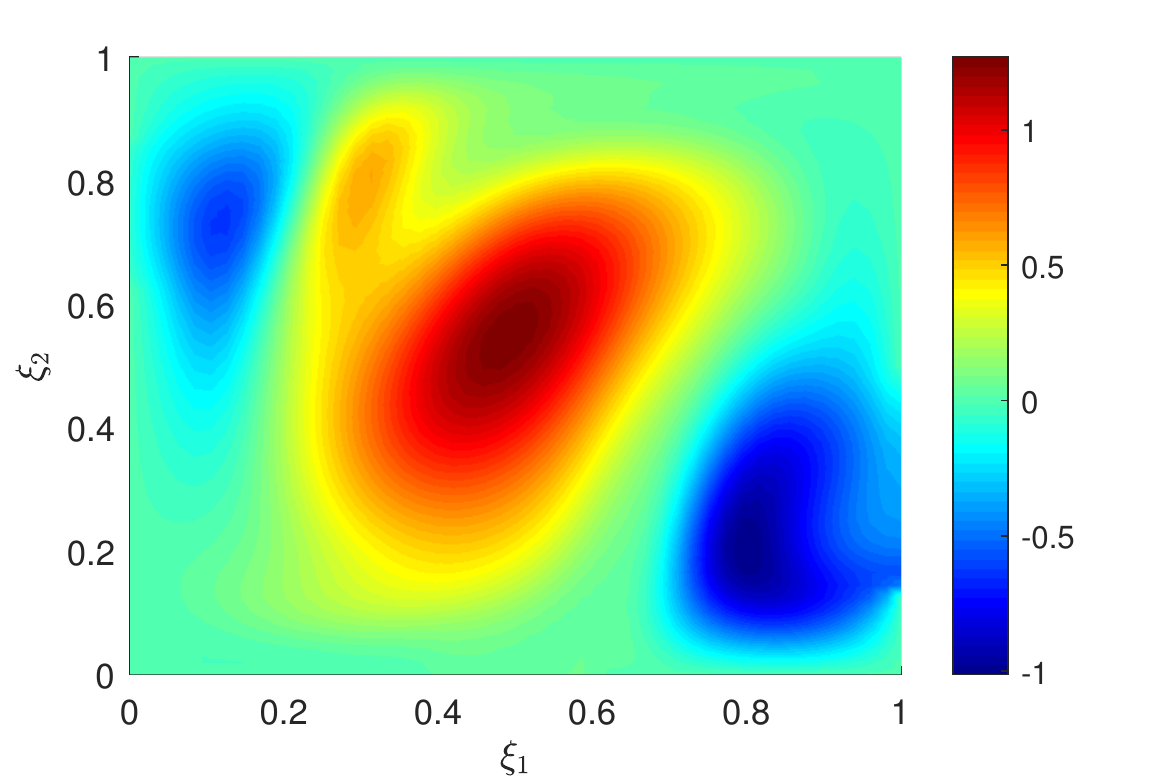}
		\captionsetup{justification=centering,font=normalsize}
		\caption[]%
		{$\theta_0$}   
		\label{fig:isth}
	\end{subfigure}
	\captionsetup{justification=centering,font=normalsize}
	\caption[]
	{ The initial state $(v_0,\theta_0)$}
	\label{fig:is}
\end{figure*}

Output tracking performance of the controller for $t \in [0,50]$ is illustrated in Figure \ref{fig:track}.
\begin{figure}[h]
	\begin{center}
		\centering
		\captionsetup{justification=centering,font=normalsize}
		\includegraphics[width=0.6\textwidth]{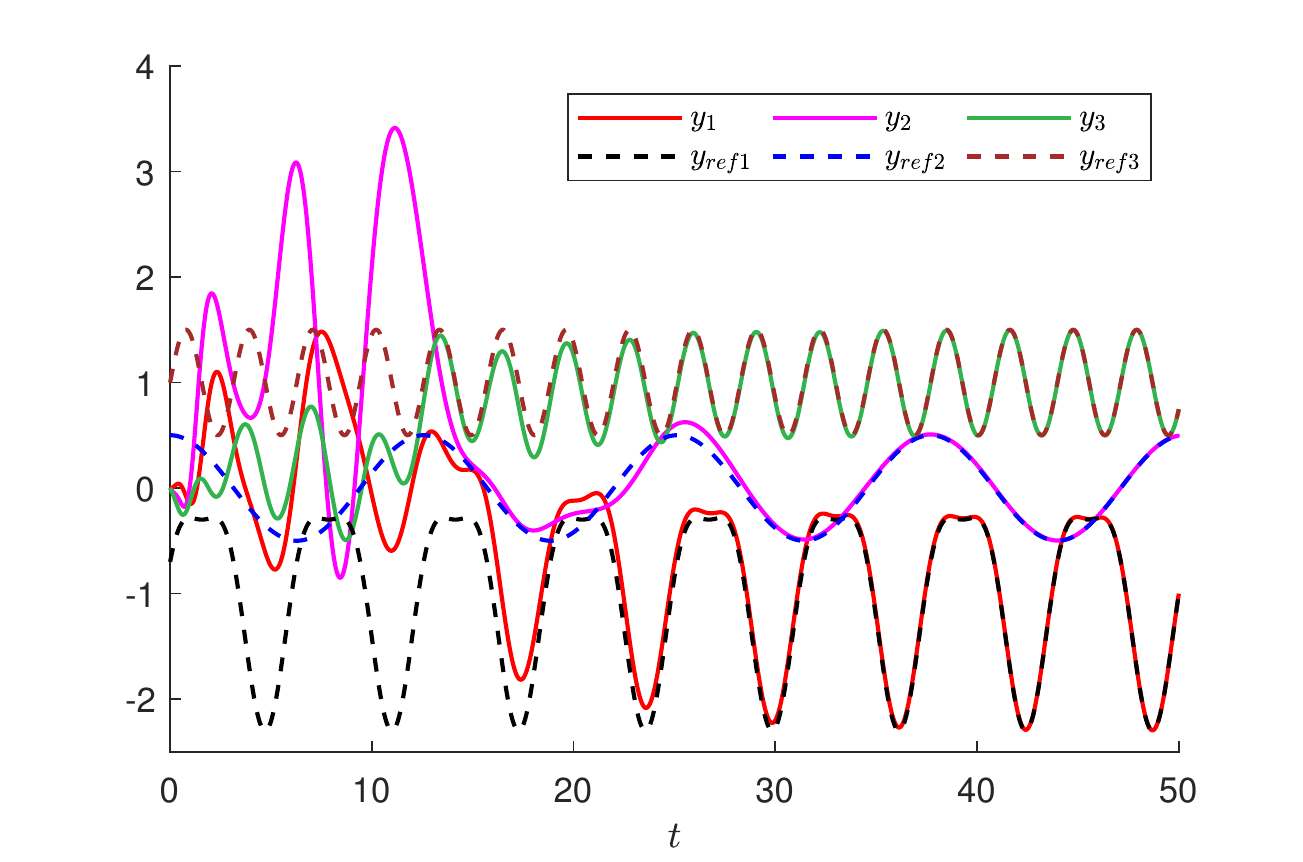}
		\caption{The system output $y = [y_1, \ y_2, \ y_3]^T$ and the reference output $y_{ref} = [y_{ref1}, \ y_{ref2}, \ y_{ref3}]^T$}
		\label{fig:track}
	\end{center}
\end{figure}
The system output converges to the reference output with accurate tracking for $t > 30$. Initial oscillations of the observation are reasonable, but the state $(v,\theta)$ of the linearized translated Boussinesq equations experiences significant oscillations. The temperature values near the controlled strips $\Gamma_I$ and $\Gamma_H$ are a prime example of this behavior, as Figure \ref{fig:state50s} suggests, while incompressibility leads to ``less localized'' velocity state.
\begin{figure*}[h!]
	\centering
	\begin{subfigure}[t]{0.48\textwidth}
		\centering
		\includegraphics[width=\textwidth]{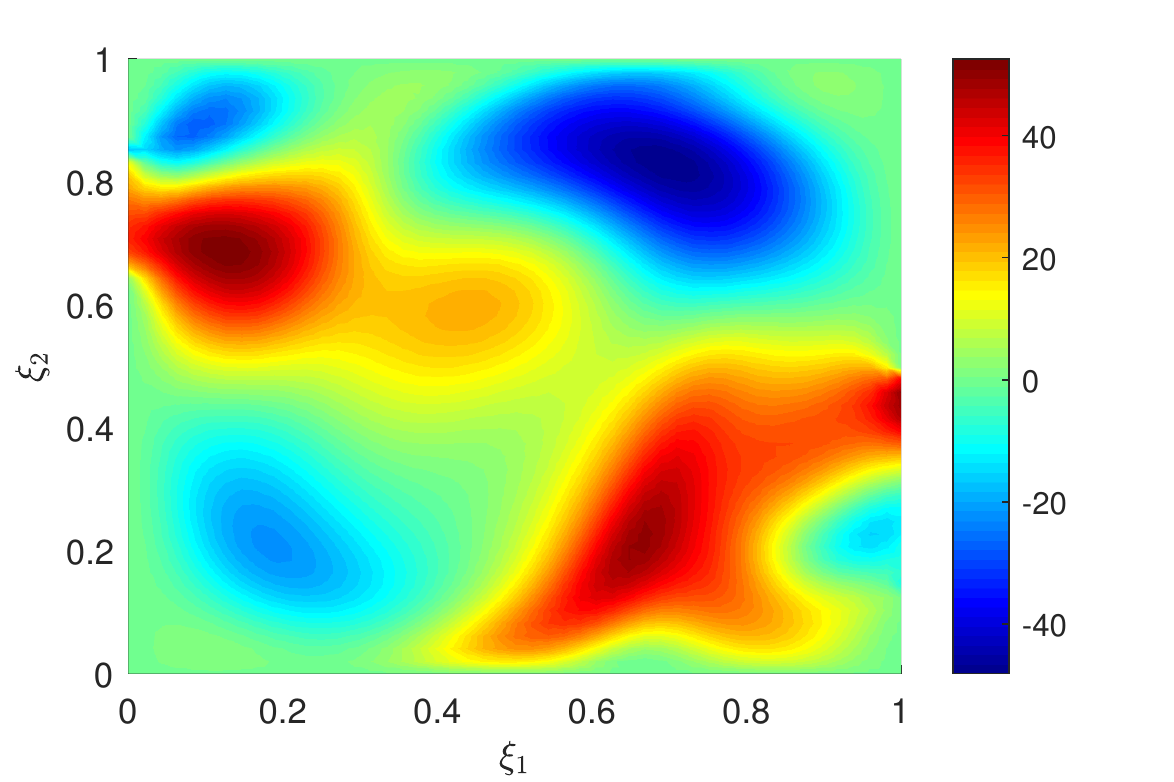}
		\captionsetup{justification=centering,font=normalsize}
		\caption[]%
		{ The velocity $v_{1}(\xi,50)$}    
		\label{fig:state50sxvel}
	\end{subfigure}
	\hfill
	\begin{subfigure}[t]{0.48\textwidth}  
		\centering 
		\includegraphics[width=\textwidth]{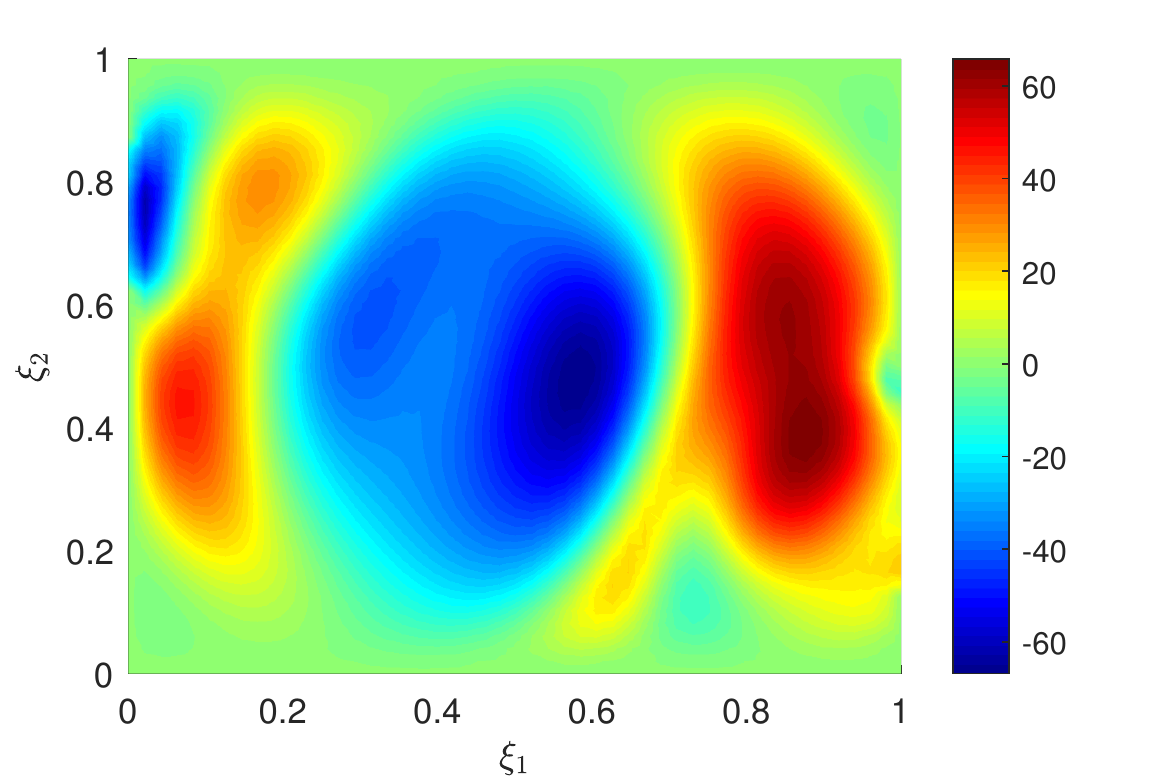}
		\captionsetup{justification=centering,font=normalsize}
		\caption[]
		{ The velocity $v_{2}(\xi,50)$}  
		\label{fig:state50syvel}
	\end{subfigure}
	\vskip\baselineskip
	\begin{subfigure}[t]{0.48\textwidth}   
		\centering 
		\includegraphics[width=\textwidth]{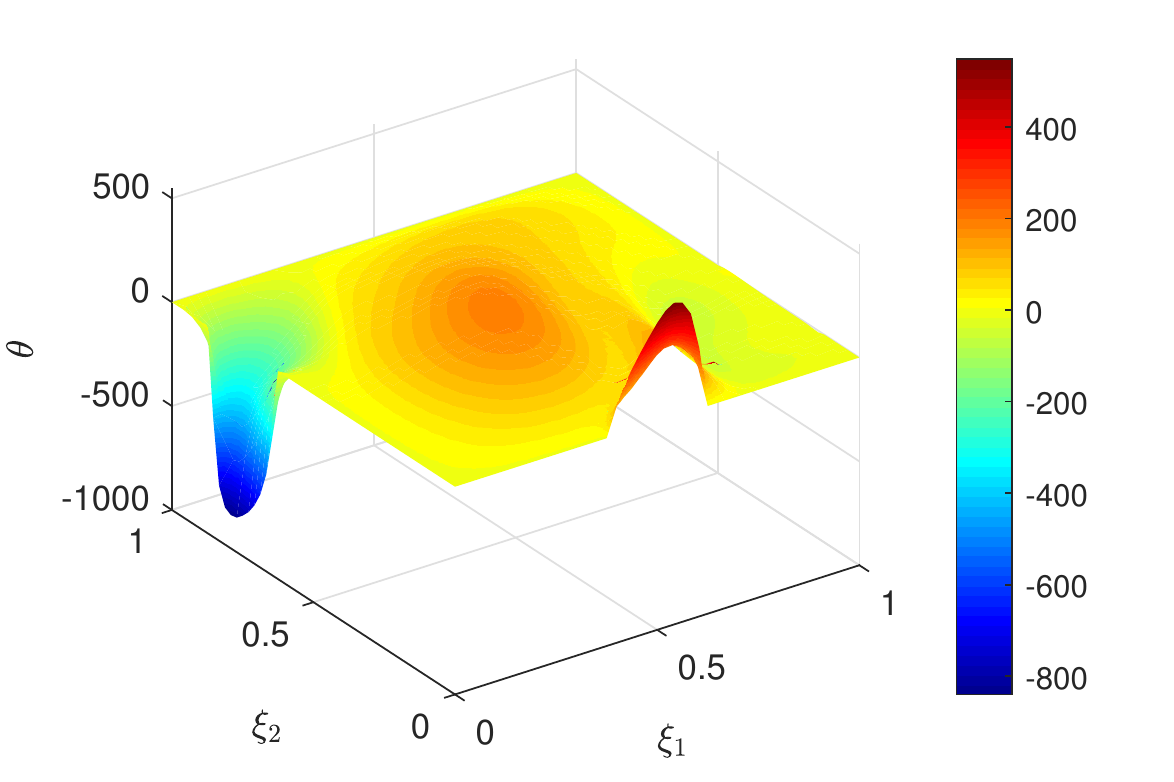}
		\captionsetup{justification=centering,font=normalsize}
		\caption[]
		{ The temperature $\theta(\xi,50)$}    
		\label{fig:state50stheta}
	\end{subfigure}
	\quad
	\begin{subfigure}[t]{0.48\textwidth}   
		\centering 
		\includegraphics[width=\textwidth]{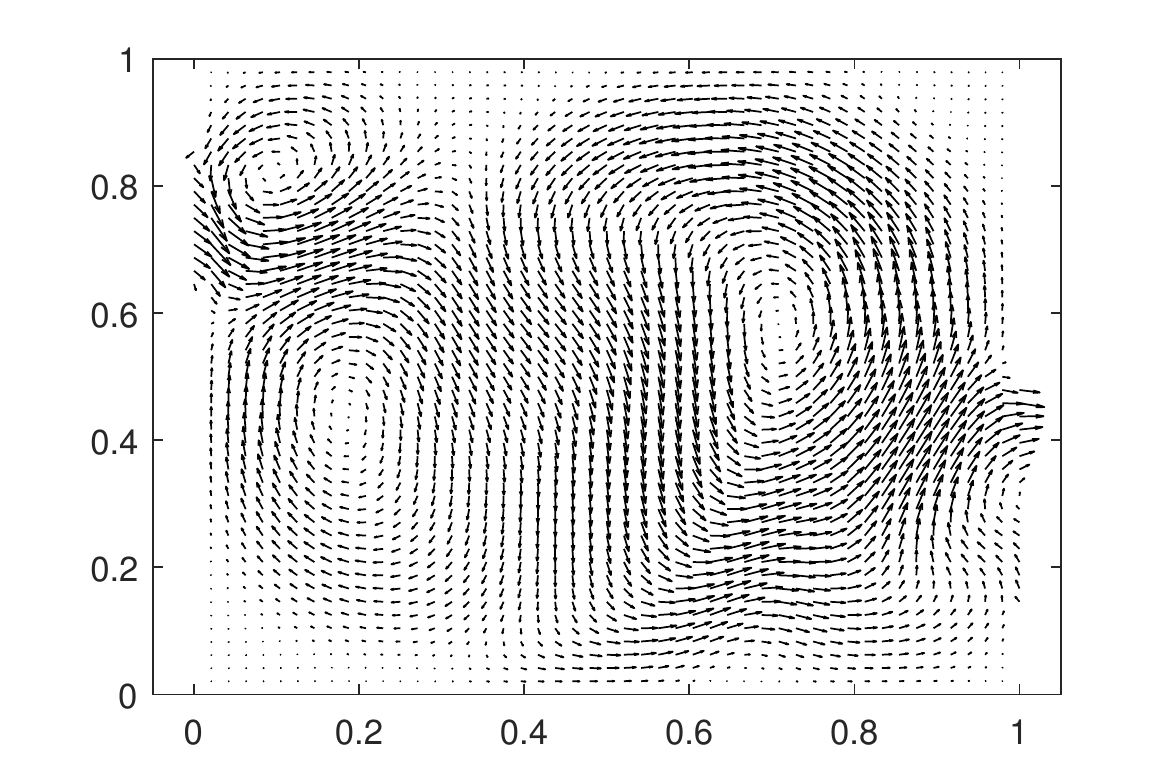}
		\captionsetup{justification=centering,font=normalsize}
		\caption[]
		{ The velocity field $v(\xi,50)$}   
		\label{fig:state50svel}
	\end{subfigure}
	\captionsetup{justification=centering,font=normalsize}
	\caption[]
	{ State $(v_1,v_2,\theta)$ of the linearized translated Boussinesq equations at the time $t=50$} 
	\label{fig:state50s}
\end{figure*}
The controlled strips actually maintain large amplitudes for both the velocity $v$ and the temperature $\theta$ throughout the simulation disregarding the initial transient, as is evident from Figure \ref{fig:Strips}.
\begin{figure*}[h]
	\centering
	\begin{subfigure}[t]{0.48\textwidth}
		\centering
		\includegraphics[width=\textwidth]{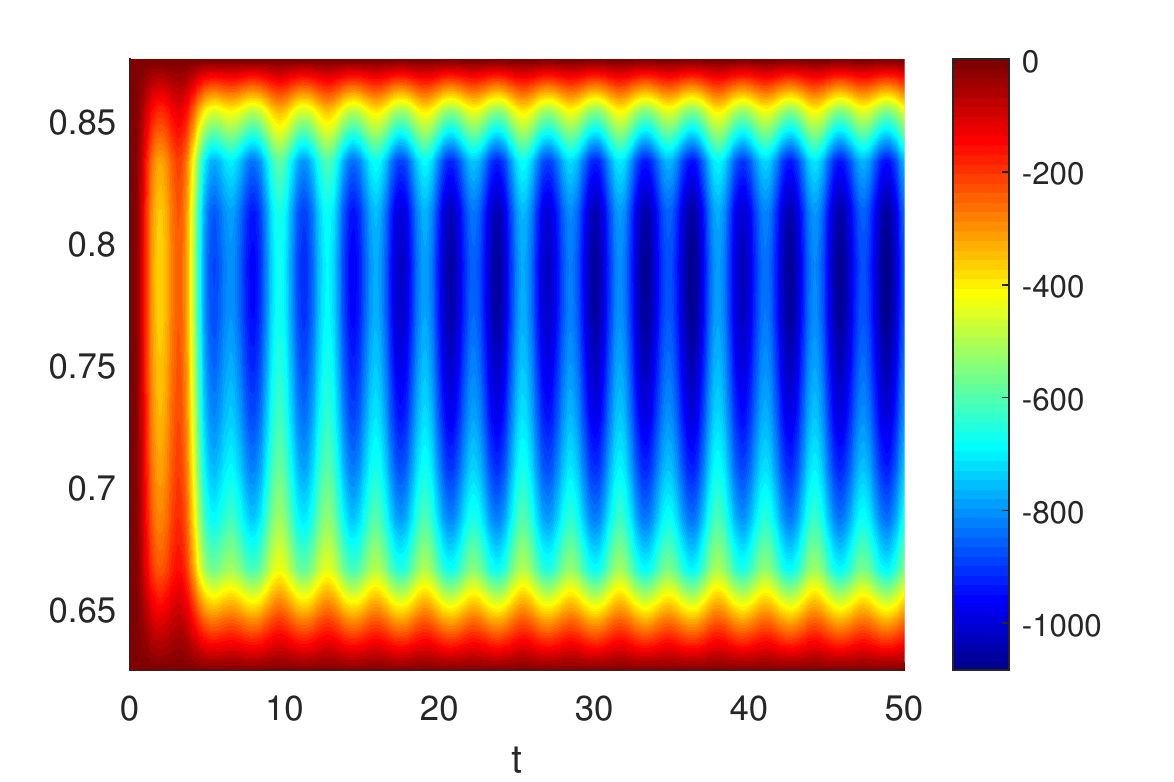}
		\captionsetup{justification=centering,font=normalsize}
		\caption[]
		{ The temperature profile within $\Gamma_I$}    
		\label{fig:Inlettemp}
	\end{subfigure}
	\hfill
	\begin{subfigure}[t]{0.48\textwidth}  
		\centering 
		\includegraphics[width=\textwidth]{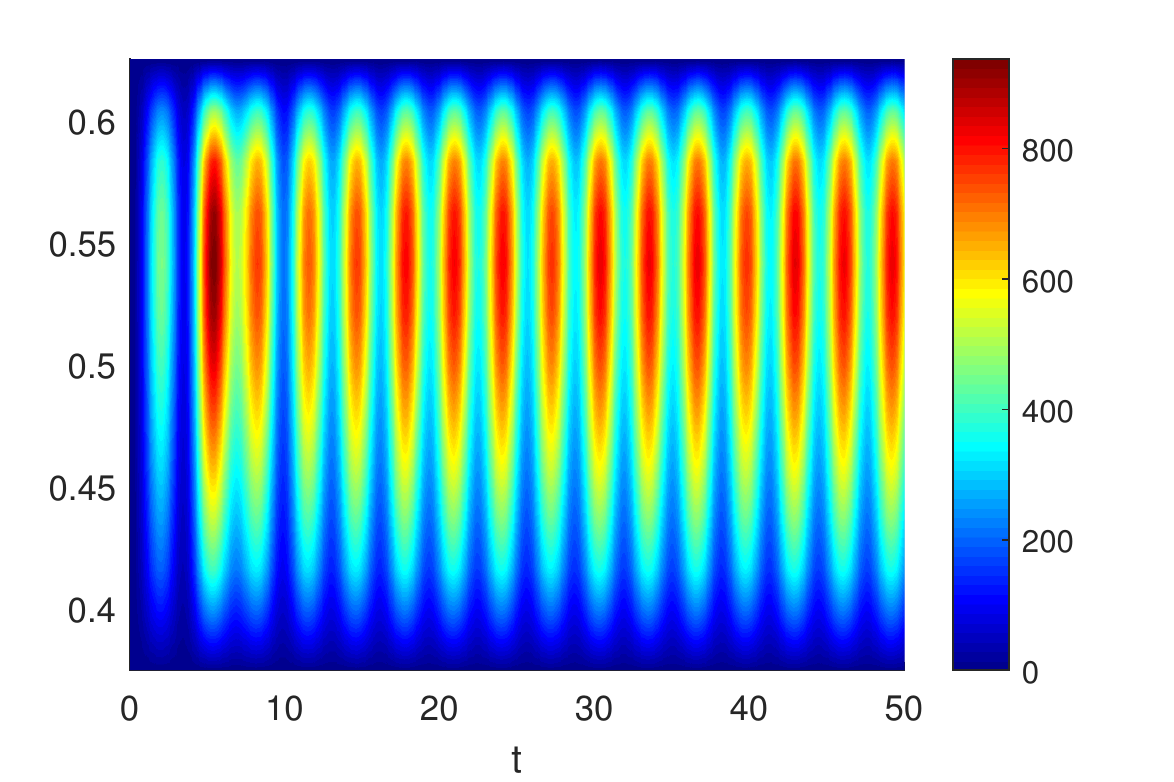}
		\captionsetup{justification=centering,font=normalsize}
		\caption[]
		{ The temperature profile within $\Gamma_H$}  
		\label{fig:Striptemp}
	\end{subfigure}
	\vskip\baselineskip
	\begin{subfigure}[t]{0.48\textwidth}   
		\centering 
		\includegraphics[width=\textwidth]{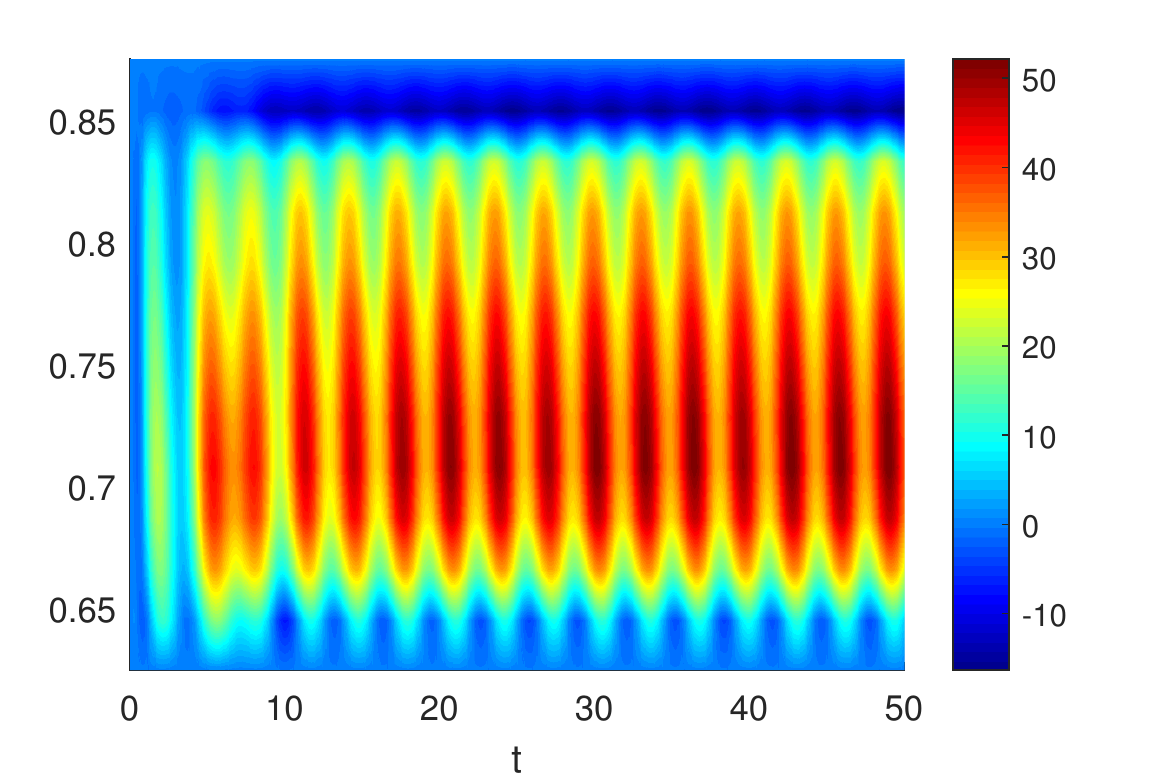}
		\captionsetup{justification=centering,font=normalsize}
		\caption[]%
		{ The $\xi_1$-velocity profile within $\Gamma_I$}    
		\label{fig:Inletxvel}
	\end{subfigure}
	\quad
	\begin{subfigure}[t]{0.48\textwidth}   
		\centering 
		\includegraphics[width=\textwidth]{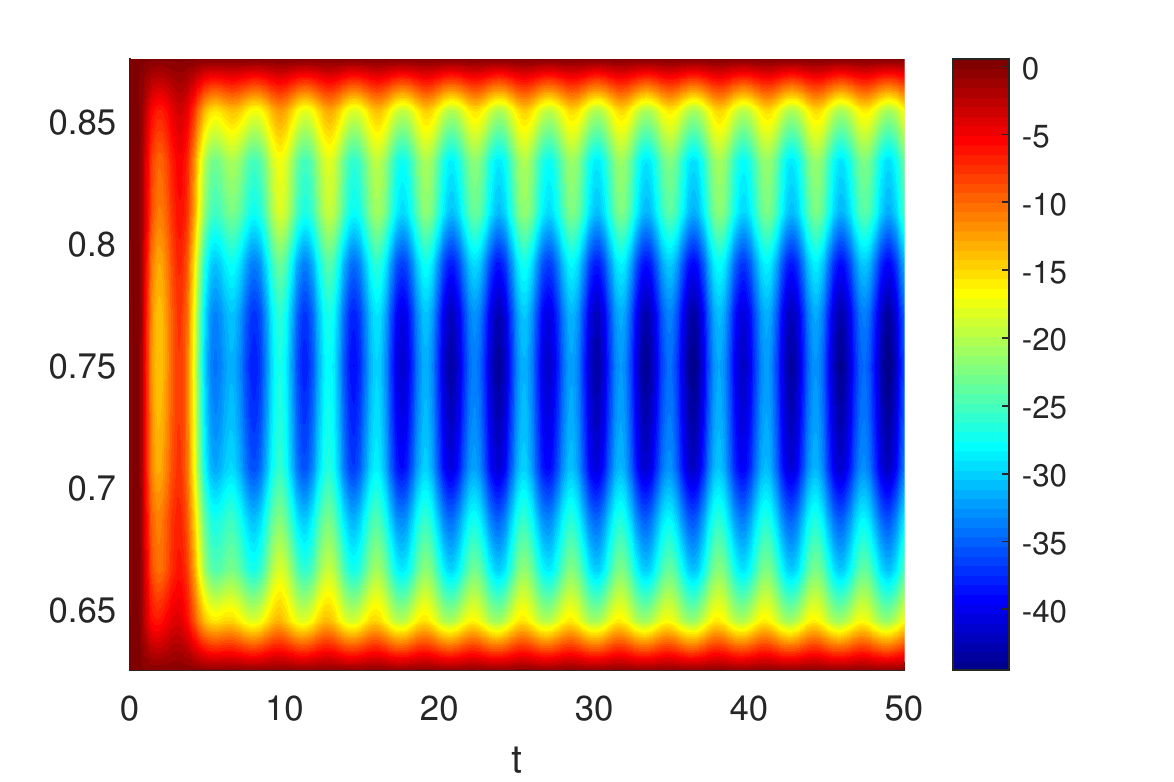}
		\captionsetup{justification=centering,font=normalsize}
		\caption[]%
		{ The $\xi_2$-velocity profile within $\Gamma_I$}    
		\label{fig:Inletyvel}
	\end{subfigure}
	\captionsetup{justification=centering}
	\caption[]
	{ State components within the controlled strips for $t \in [0,50]$} 
	\label{fig:Strips}
\end{figure*}
Similarly the boundary inputs to the room do not change sign after the transient, see Figure \ref{fig:Inputs}, although the plant input component $u_3$ just barely does.

\begin{figure*}[h]
	\centering
	\begin{subfigure}[t]{0.48\textwidth}
		\centering
		\includegraphics[width=\textwidth]{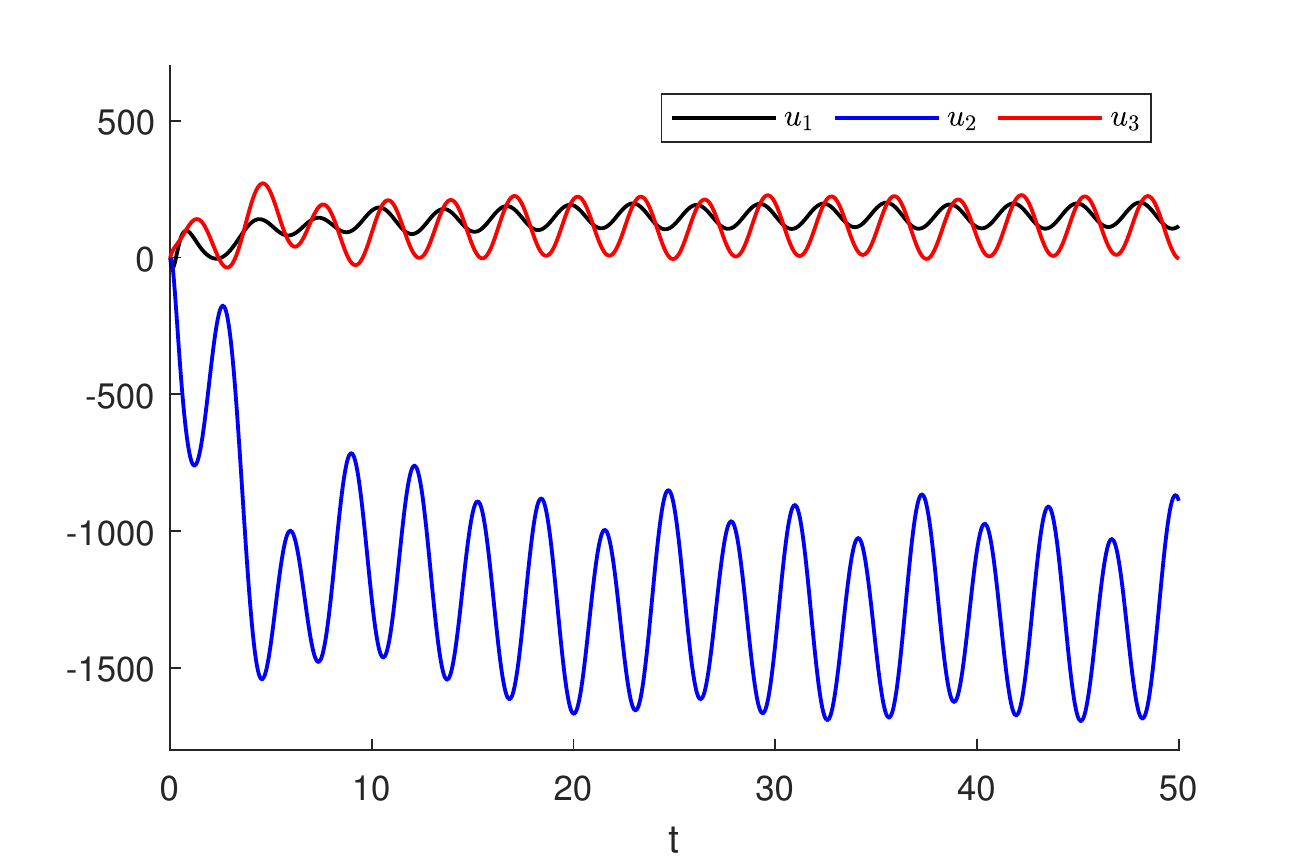}
		\captionsetup{justification=centering,font=normalsize}
		\caption[]
		{ The plant input $u = [u_1, \ u_2, \ u_3]^T$}    
		\label{fig:Input}
	\end{subfigure}
	\hfill
	\begin{subfigure}[t]{0.48\textwidth}  
		\centering 
		\includegraphics[width=\textwidth]{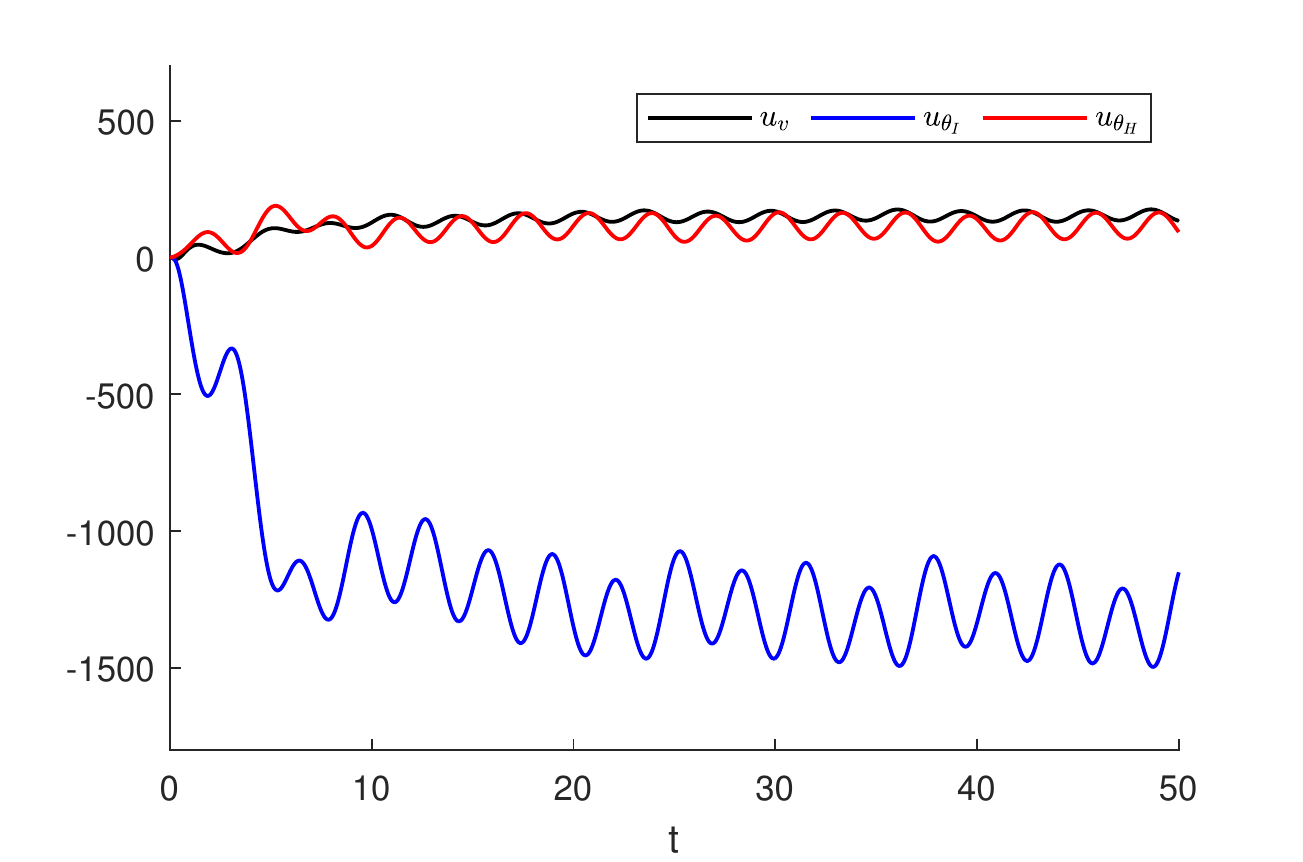}
		\captionsetup{justification=centering,font=normalsize}
		\caption[]
		{ The boundary input $u_b = [u_v, \ u_{\theta_I}, \ u_{\theta_H}]^T$}  
		\label{fig:BCSinput}
	\end{subfigure}
	\captionsetup{justification=centering}
	\caption[]
	{ Plant input $u$ and the corresponding boundary input $u_b$ generated by the actuator}
	\label{fig:Inputs}
\end{figure*}

%% file: Conclusion.tex
\section{Conclusion}\label{Sec:conclusion}

We studied robust temperature and velocity output tracking of a two-dimensional room model with the fluid dynamics governed by the linearized Boussinesq equations. For the room model, we considered the natural case of the control applied on the fluid via the boundary and the observations performed on the fluid at the boundary. Related to the control and observation operations, we modeled actuator and sensor dynamics of the system. These additional models lead to more complex system dynamics while improving mathematical properties of the control and observation operators and possibly improving the model's accuracy. We examined effects of the added actuator and sensor dynamics on the system properties such as exponential stabilizability and exponential detectability and implemented an internal model based error feedback controller design for robust temperature and velocity output tracking for the room model.
In addition to being robust, the controller is suitable for unstable systems, requires only output information instead of full state information and is of low order for efficient applicability.

As an example, we illustrated robust output tracking of the linearized Boussinesq equations with actuator and sensor dynamics in the case of three boundary controls, a mix of one boundary and two in-domain observations and a boundary disturbance, each affecting either a velocity or the temperature component of the system. The controller achieved accurate tracking with relatively small transient observation oscillation, although the system state reaches large absolute values locally. Analogous approach of actuator and sensor modeling can be applied for robust output tracking of a class of linear systems with boundary control and observation.